\documentclass[11pt]{amsart}
\pdfoutput=1
\usepackage{amsmath,amssymb,amsthm}
\usepackage{enumitem}
\usepackage{tikz}
\usepackage{tikz-cd}
\usetikzlibrary{decorations.pathmorphing}
\usepackage{bbm}
\usepackage{array}
\usepackage{hyperref}
\hypersetup{
    colorlinks=true,
    linkcolor=blue,
    filecolor=magenta,      
    urlcolor=blue}
\usepackage{biblatex}
\bibliography{bibliography}

\theoremstyle{plain}
\newtheorem{theorem}{Theorem}
\newtheorem{lemma}[theorem]{Lemma}
\newtheorem{corollary}[theorem]{Corollary}
\newtheorem{proposition}[theorem]{Proposition}

\numberwithin{theorem}{section}

\theoremstyle{definition}
\newtheorem{definition}[theorem]{Definition}

\newtheorem{remark}[theorem]{Remark}
\newtheorem{assumption}[theorem]{Assumption}

\DeclareMathOperator{\GL}{GL}
\DeclareMathOperator{\im}{im}

\DeclareMathOperator{\rank}{rank}

\DeclareMathOperator{\Hom}{Hom}
\DeclareMathOperator{\Aut}{Aut}

\DeclareMathOperator{\Spec}{Spec}

\DeclareMathOperator{\Pic}{Pic}

\DeclareMathOperator{\vir}{vir}
\DeclareMathOperator{\CR}{CR}
\DeclareMathOperator{\amb}{amb}
\DeclareMathOperator{\ct}{ct}
\DeclareMathOperator{\cs}{cs}
\DeclareMathOperator{\Eff}{Eff}
\DeclareMathOperator{\age}{age}

\newcommand{\C}{\mathbb{C}}

\newcommand{\Q}{\mathbb{Q}}

\newcommand{\CP}{\mathbb{P}}

\newcommand{\id}{\ensuremath{\mathrm{id}}}
\newcommand{\sslash}{\mathbin{/\mkern-6mu/}}

\newtheorem*{rep@theorem}{\rep@title}
\newcommand{\newreptheorem}[2]{%
\newenvironment{rep#1}[1]{%
 \def\rep@title{#2 \ref{##1}}%
 \begin{rep@theorem}}%
 {\end{rep@theorem}}}

\newreptheorem{theorem}{Theorem}
\newreptheorem{lemma}{Lemma}
\newreptheorem{proposition}{Proposition}
\newreptheorem{definition}{Definition}
\newreptheorem{corollary}{Corollary}

\title{Quantum Serre duality for quasimaps}

\author[Heath]{Levi Heath}
\address{
  \begin{tabular}{l}
   Levi Heath \\
   \hspace{.1in} Colorado State University \\
      \hspace{.1in} Department of Mathematics \\
   \hspace{.1in} 1874 Campus Delivery, Fort Collins, CO, USA, 80523-1874\\
   \hspace{.1in} Email: {\bf levi.heath@colostate.edu} \\
  \end{tabular}
}

\author[Shoemaker]{Mark Shoemaker}
\address{
  \begin{tabular}{l}
   Mark Shoemaker \\
   \hspace{.1in} Colorado State University \\
      \hspace{.1in} Department of Mathematics \\
   \hspace{.1in} 1874 Campus Delivery, Fort Collins, CO, USA, 80523-1874\\
   \hspace{.1in} Email: {\bf mark.shoemaker@colostate.edu} \\
  \end{tabular}
}

\begin{document}


\maketitle

\textbf{Abstract}: Let $X$ be a smooth variety or orbifold and let $Z \subseteq X$ be a complete intersection defined by a section of a vector bundle $E \to X$.  Originally proposed by Givental, quantum Serre duality refers to a precise relationship between the Gromov--Witten invariants of $Z$ and those of the dual vector bundle $E^\vee$.  In this paper we prove a quantum Serre duality statement for quasimap invariants.  In shifting focus to quasimaps, we obtain a comparison which is simpler and which also holds for non-convex complete intersections. By combining our results with the wall-crossing formula developed by Zhou, we recover a quantum Serre duality statement in Gromov-Witten theory without assuming convexity.

\tableofcontents


\section{Introduction}

Let $X$ be a smooth projective variety or Deligne--Mumford stack and let $Z \subset X$ be a smooth complete intersection, defined by the vanishing of a section of a vector bundle $E \to X$.  Quantum Serre duality refers to a relationship between the genus-zero Gromov-Witten invariants of $Z$ and those of the dual vector bundle $E^\vee$.  In this paper we investigate this correspondence in the context of quasimap invariants.

\subsection{History}

Quantum Serre duality was first described in mathematics by Givental in \cite{Givental_Equivariant_Gromov_Witten_invariants_1996}, for the case of $X= \CP^n$.  The correspondence is given as a relation between generating functions of genus-zero Gromov--Witten invariants of $Z$ and $E^\vee$ via a complicated change of variables and a non-equivariant limit.  Since then, quantum Serre duality has been generalized and reformulated in a number of ways.  In 
 \cite{Coates_Givental_Quantum_Riemann-Roch_Lefschetz_and_Serre_2007}, Coates--Givental employ Givental's symplectic formalism \cite{givental2004symplectic} to show that twisted overruled Lagrangian cones for $E$ and $E^\vee$ may be identified by a symplectic isomorphism.  When $E$ is \emph{convex}, this implies that generating functions of Gromov--Witten invariants of $Z$ may be recovered from those of $E^\vee$.  In 
 \cite{iritani2016quantum}, Iritani--Mann--Mignon observe that quantum Serre duality may be cast as an isomorphism of quantum $D$-modules, and is compatible with a Fourier--Mukai transform in $K$-theory.  This result was refined and extended to orbifolds by the second author in \cite{Shoemaker_2018}.  

The utility of the correspondence is based on the fact that in many cases the geometry of $E^\vee$ is simpler than that of $Z$.  For instance, if $X$ is a toric variety then $E^\vee$ is as well.  It was noted in \cite{Coates_Givental_Quantum_Riemann-Roch_Lefschetz_and_Serre_2007} that by using quantum Serre duality, the mirror theorem for $Z$ follows from the mirror theorem for $E^\vee$.  In recent years, quantum Serre duality has been employed to prove other correspondences in Gromov--Witten theory.  Applications have included the crepant transformation conjecture \cite[\S 6]{coates2018crepant}, the LG/CY correspondence \cite{priddis2014proof, shoemaker2020integral}, and the Gromov--Witten theory of extremal transitions \cite{mi2020extremal}.

A theme which persists through all of the formulations of quantum Serre duality described above is that 
in order to observe a correspondence, the Gromov--Witten invariants of $Z$ and $E^\vee$ must be packaged in a clever way (Lagrangian cones, $D$-modules, etc...).  There is no simple relation between the individual Gromov--Witten invariants of $Z$ and those of $E^\vee$.

\subsection{Quasimaps}
Let $W$ be an affine variety, acted on by a reductive algebraic group $G$, and let $\theta$ be a character of $G$ such that $W^{ss}(\theta) = W^s(\theta)$.  Denote by $X$ the GIT stack quotient $[W^{ss}(\theta) /G]$.
The moduli stacks $Q^{0+}_{g,k}(X, \beta)$ of stable quasimaps to $X$ (which depends implicitly on the GIT presentation)
provide an alternative to Kontsevich's space of stable maps. 
Generalizing the stable quotient spaces of \cite{marian2011moduli}, quasimaps were first introduced for toric varieties in \cite{ciocan2010moduli} and generalized to GIT quotients in  \cite{Ciocan_Fontanine_Kim_Maulik_Stable_quasimaps_to_GIT_quotients_2014,Cheong-Ciocan-Fontanine-Kim_2015}.

In contrast to stable maps, a quasimap $f: C \dashrightarrow X$ generally defines only a rational map.  For $X = [W^{ss}(\theta) / G]$, a quasimap to $X$ is a morphism from a (orbi-)curve $C$ to the stack $[W/G]$ such that the preimage of the unstable locus is a finite set of points, disjoint from the nodes and markings of $C$.  Under certain mild conditions on $W$, $G$, and $\theta$, the moduli space of ($0+$-stable) quasimaps $Q^{0+}_{g,k}(X, \beta)$ is (relatively) proper, and carries a canonical virtual fundamental class (see \S\ref{Section: Quasimaps to a GIT stack quotient} for details).  As in Gromov--Witten theory, quasimap invariants may be defined by integrating over the virtual fundamental class.

The relationship between quasimap invariants and Gromov--Witten invariants is now well understood in many cases; these results are called $\varepsilon$-wall-crossing \cite{Ciocan_Fontanine_Kim_Wall_crossing_in_genus_zero_quasimap_theory_and_mirror_maps_2014, ciocan2020quasimap, clader2017higher, zhou2020quasimap}.  In principal, one can determine Gromov--Witten invariants from quasimap invariants and vice versa.  

\subsection{Results}
Let $X = [W^{ss}(\theta) / G]$ be as in the previous section.  Denote by $\mathfrak{X}$ the stack quotient $[W/G]$.  A choice of representation $\tau$ in $\Hom(G, \GL(r,\C))$ determines vector bundles 
\begin{align*}
    E &:= [W^{ss}(\theta) \times \C^r /G]  \to X   \\
    \mathcal{E} &:= [W \times \C^r /G]  \to \mathfrak{X}.   
\end{align*}
 Assume $\mathcal{E}$ is \emph{weakly convex} (Definition~\ref{Definition: weak convexity}).  Let $s \in \Gamma(X, E)$ be a section of $E$ defined as in \S\ref{Section: A local complete intersections in X}, and let $Z = Z(s) \subset X$ be the complete intersection defined by $s$.

In this paper we compare the two-pointed genus-zero quasimap invariants of $Z$ with those of $E^\vee$.  In contrast to the case of  Gromov--Witten theory, here we obtain a direct relation between invariants, as well as a statement at the level of virtual classes.

The first step is to identify the relevant state spaces.  Consider the diagram
\begin{equation*}
    \begin{tikzcd}
        & IE^\vee \arrow[d, swap, "p"] \\
        IZ \arrow[r, "j"] & IX \ar[u, bend right, swap, "i"]
    \end{tikzcd}
\end{equation*}
of inertia stacks. Define the \emph{ambient cohomology} of $Z$ by 
\begin{equation*}
    H^*_{\CR,\amb}(Z) := \im(j^*)
\end{equation*}
and the \emph{cohomology of compact type} of $E^\vee$ by 
\begin{equation*} 
    H^*_{\CR,\ct}(E^\vee) := \im(i_*).
\end{equation*}
Under mild assumptions, there exists (Lemma~\ref{Lemma: Delta_+ is an isomorphism}) an isomorphism 
\begin{equation*}
    \tilde\Delta: H^*_{\CR,\ct}(E^\vee) \to H^*_{\CR,\amb}(Z),
\end{equation*}
characterized by the fact that $\tilde \Delta(i_*(\alpha)) = e^{\pi i \age_{g}(\mathcal{E})} j^*(\alpha)$ for $\alpha\in H^*(X_g)$. Our first result is that $\tilde \Delta$ identifies two-pointed genus-zero quasimap invariants of $Z$ and $E^\vee$ up to a sign.

\begin{reptheorem}{Theorem: Two-pointed invariant relationship}
Given elements $\gamma_1,\gamma_2\in H^*_{\CR,\ct}(E^\vee)$, we have the equality 
\begin{equation*}
    \left\langle \tilde\Delta\left(\gamma_1\right) \psi_1^{a_1}, \tilde\Delta\left(\gamma_2\right) \psi_2^{a_2} \right\rangle^{Z,0+}_{0,\beta}
    =
    e^{\pi i(\beta(\det \mathcal{E})+\rank(E))} \left\langle \gamma_1 \psi_1^{a_1}, \gamma_2 \psi_2^{a_2} \right\rangle^{E^\vee,0+}_{0,\beta} .
\end{equation*}
\end{reptheorem}

Our focus on two-pointed invariants arises from their role in solutions to the quantum differential equation. We use the superscripts $\infty$ and $0+$ to distinguish between Gromov--Witten theory and quasimap theory respectively.  Recall the Dubrovin connection in Gromov--Witten theory
\begin{equation*}
    \nabla_i^{X,\infty} f := \frac{\partial}{\partial t_i} f + \frac{1}{z} T^i \bullet^{X,\infty}_{\boldsymbol{t}} f ,
\end{equation*}
where $\{T^i\}_{i \in I}$ is a basis of $H^*_{CR}(X)$ and $\bullet^{X,\infty}_{\boldsymbol{t}}$ denotes the quantum product.
The fundamental solution is given by 
\begin{equation*}
    L^{X,\infty}(\boldsymbol{t},z)(\alpha) := \alpha + \sum_{i\in I} \left\langle\left\langle \frac{\alpha}{-z-\psi} , T_i \right\rangle\right\rangle^{X,\infty}_0 (\boldsymbol{t}) T^i ,
\end{equation*}
where the double bracket is defined in Definition~\ref{dinvt}.

In quasimap theory, one can define an analogous product $\bullet^{X, 0+}_{\boldsymbol{t}}$ and connection $\nabla^{X, 0+}$, replacing Gromov--Witten invariants with quasimap invariants in the definition.  As in Gromov-Witten theory, $\nabla^{X, 0+}$ is flat and its solution is given by $L^{X,0+}(\boldsymbol{t},z)$.
Restricting to $\boldsymbol{t} = 0$, we obtain the generating function
\begin{equation}
\tag{\ref{Equation: Fundamental Solution}}
    L^{X,0+}(z)(\alpha) := L^{X,0+}(\boldsymbol{0},z)(\alpha) = \alpha + \sum_{i\in I} \sum_{\beta\in\Eff} q^\beta \left\langle \frac{\alpha}{-z-\psi} , T_i \right\rangle^{X,0+}_{0,\beta} T^i ,
\end{equation}
of two-pointed genus-zero quasimap invariants.
Our main theorem is an identification of $L^{Z,0+}(z)$ and $L^{E^\vee,0+}(z)$.

\begin{reptheorem}{Theorem: QSD for quasimaps}
The transformation $\tilde\Delta$ identifies the operators $L^{Z,0+}(z)$ and $L^{E^\vee,0+}(z)$ up to a change of variables in the Novikov parameter:
\begin{equation*} 
    L^{Z,0+}(z) \circ \tilde\Delta
    =
    \tilde\Delta \circ L^{E^\vee,0+}(z) |_{q^\beta\mapsto e^{\pi i \beta(\det \mathcal{E})} q^\beta} .
\end{equation*}
\end{reptheorem}

This theorem provides a quasimap analogue of \cite[Proposition~6.13]{Shoemaker_2018}, proven by the second author in the context of Gromov--Witten theory.  Note that here the comparison statement is simpler and holds for more general vector bundles $E^\vee\to X$.


Our strategy of proof is different than the arguments used to prove similar results in Gromov--Witten theory.  By working with quasimaps as opposed to stable maps, we are able to recover a direct relationship between virtual classes. As in Gromov--Witten theory, we first relate the virtual fundamental classes  $[Q^{0+}_{0,2}(Z,\beta)]^{\vir}$ and  $[Q^{0+}_{0,2}(E^\vee,\beta)]^{\vir}$ to certain twisted virtual classes on  $Q^{0+}_{0,2}(X,\beta)$. We then show, using (non-quantum) Serre duality, that these twisted virtual classes agree up to a sign (Theorem~\ref{Theorem: Cycle-valued QSD}).  The cycle-valued statement is interesting in its own right, as its failure in Gromov--Witten theory is exactly what accounts for the change of variables appearing in previous results.

In the last section, we combine our main theorem with the recent $\varepsilon$-wall-crossing results of Zhou \parencite[Theorem~1.12.2]{zhou2020quasimap} to recover a quantum Serre duality statement in Gromov--Witten theory when $E^\vee$ is a GIT quotient.

\begin{repcorollary}{Corollary: GW QSD}
The operator $\tilde \Delta$ identifies the fundamental solutions $L^{Z, \infty}$ and $L^{E^\vee, \infty}$ up to a change of variables:
\begin{align*}
    &L^{Z,\infty}(\mu^{Z,\geq 0+}(q,-\psi),z) \circ \tilde \Delta \\ 
    =& \tilde \Delta \circ L^{E^\vee,\infty}(\mu^{E^\vee,\geq 0+}(q,-\psi),z) |_{q^\beta\mapsto e^{\pi i \beta(\det \mathcal{E})} q^\beta} ,
\end{align*}
where $\mu^{Z,\geq 0+}(q,-\psi),z)$ and $\mu^{E^\vee,\geq 0+}(q,-\psi),z)$ are  changes of variables defined in terms of the $I$-functions for $Z$ and $E^\vee$ respectively \parencite[\S 1.11]{zhou2020quasimap}.
\end{repcorollary}

We may interpret Theorem~\ref{Theorem: QSD for quasimaps} and Corollary~\ref{Corollary: GW QSD} together as evidence that quantum Serre duality arises naturally in the setting of quasimaps, at least when the base $X$ can be expressed as a GIT quotient.  The complicated change of variables appearing in previous results in Gromov--Witten theory is not inherent to the correspondence itself, but rather a remnant of the $\varepsilon$-wall-crossing formula arising in the passage from $0+$-stable quasimaps to $\infty$-stable maps.  A further benefit to this approach is that we no longer require $E$ be convex, as we describe below.

\subsection{Removing the convexity hypothesis}
In previous results on quantum Serre duality in Gromov--Witten theory, the vector bundle $\mathcal{E} \to \mathfrak{X}$ was required to be convex.  In the case that $X$ is a variety this holds whenever $\mathcal{E}$ is \emph{semi-positive} (Definition~\ref{Definition: Semi-positive}).  However, it was observed in \cite{coates2012quantum} that when $X$ is an orbifold, semi-positivity of $\mathcal{E}$ does not imply convexity.  This has been a serious obstacle in the computation of Gromov--Witten invariants of orbifold complete intersections.  By working with quasimap invariants, we avoid the convexity requirement.  Our results require only \emph{weak convexity} (Definition~\ref{Definition: weak convexity}), which once again holds whenever $\mathcal{E}$ is
semi-positive (and in fact holds slightly more generally, see Definition~\ref{Definition: Weakly semi-positive}).  

Recent works by Gu\'er\'e \cite{guere2020congruences} and Wang \cite{wang2019mirror} also address the computation of genus-zero Gromov--Witten theory of complete intersections when $E$ is not convex.  Wang's work also features quasimaps, and is based on a similar philosophy to that used in this paper.

\subsection{Acknowledgements}

The first author would like to thank Adam Afandi, Vance Blankers, Renzo Cavalieri, and Andy Fry for many helpful discussions on Gromov-Witten theory, toric varieties, and orbifolds. The authors would also like to thank the MMARGS reading seminar at Colorado State University. This work was partially supported by NSF grant DMS-1708104.


\section{Preliminaries}
\label{Section: Preliminaries}
In this section we define quasimap invariants and introduce the generating functions which we will use in the rest of the paper.  This section also serves to set notation.

\subsection{Orbifold cohomology} 
\label{Section: Orbifold cohomology}

We refer the reader to \cite{Adem_Leida_Ruan_Orbifolds_and_stringy_topology_2007} for an introduction to Chen-Ruan orbifold cohomology.  Denote by $X$ an oriented orbifold which admits a finite good cover.

\begin{definition}
Define the \emph{inertia stack} of $X$, written $IX$, by the fiber square
\begin{center}
\begin{tikzcd}
    IX \arrow[r] \arrow[d]
    & X \arrow[d,"\Delta"] \\
    X \arrow[r,"\Delta"]
    & X \times X .
\end{tikzcd}
\end{center}
\end{definition}

For a stack quotient $X=[U/G]$ the inertia stack is composed of a disjoint union of suborbifolds $X_g=[U_g/C(g)]$ where $U_g$ are the elements of $U$ fixed by $g\in G$.  Let $S$ be a set of representatives of each conjugacy class of $G$, then we may write the inertia stack as
\begin{equation*}
    IX = \bigsqcup_{g\in S} X_g .
\end{equation*}
We refer to each $X_g$ as a \emph{twisted sector} and call the twisted sector corresponding to the identity the \emph{untwisted sector}.

Let $(x,g)$ be a point in a twisted sector $X_g$.  The tangent space $T_x X$ splits as the direct sum of eigenspaces
\begin{equation*}
    T_x X = \bigoplus_{0\leq f<1} (T_x X)_f ,
\end{equation*}
where $g$ acts on the fiber $(T_x X)_f$ by multiplication by $e^{2\pi i f}$. The \emph{age shift} for $X_g$ is defined to be
\begin{equation*}
    \iota_g = \sum_{0\leq f <1} f \dim_\C (T_x X)_f .
\end{equation*}

\begin{definition}
\cite{Chen-Ruan_2004,Adem_Leida_Ruan_Orbifolds_and_stringy_topology_2007}
The \emph{$d$th Chen-Ruan cohomology group} of $X$ is
\begin{equation*}
    H^d_{\CR}(X) := \bigoplus_{g\in S} H^{d-2\iota_g} (X_g;\C)
\end{equation*}
and
\begin{equation*}
    H^*_{\CR}(X) := \bigoplus_{d\in \Q^{0+}_{\geq 0}} H^{d}_{\CR} (X).
\end{equation*}
\end{definition}

Unless specified otherwise we assume cohomology groups have complex coefficients.  Let $\iota$ be the involution of $IX$ that maps $X_g$ to $X_{g^{-1}}$.  Denote the compactly supported Chen-Ruan cohomology of $X$ by $H^*_{\CR, \cs}(X)$. 

\begin{definition} \parencite[\S 4]{Adem_Leida_Ruan_Orbifolds_and_stringy_topology_2007}
The \emph{Chen-Ruan Poincar\'e pairing} of the classes $\alpha\in \nolinebreak H^*_{\CR}(X)$ and $\beta\in H^*_{\CR,\cs}(X)$ is given by the integral
\begin{equation*}
    \langle \alpha, \beta \rangle^X := \int_{IX} \alpha \cup \iota^*\beta .
\end{equation*}
\end{definition}

\subsection{GIT stack quotients}
\label{Section: GIT stack quotients}

We follow the notations and definitions developed in \cite{Cheong-Ciocan-Fontanine-Kim_2015}.  Let $W$ be an irreducible affine variety with a right action by a reductive algebraic group $G$.  Let $\theta$ be a character of $G$ and $\C_\theta$ be the corresponding one-dimensional $G$-representation.  We also denote the linearization $W\times\C_\theta$ as $\C_\theta$.  Denote the semistable locus by $W^{ss}(\theta)$ and the stable locus by $W^s(\theta)$.  Observe the following diagram of quotients:

\begin{center}
\begin{tikzcd}
    X:=[W^{ss}(\theta)/G] \arrow[r,hook] \arrow[d]
    & \mathfrak{X}:=[W/G] \arrow[d] \\
    \underline{X}:=W \hspace{-5pt} \sslash_{\hspace{-2pt}\theta} \hspace{-3pt} G \arrow[r]
    &\underline{X}_0:=\Spec(\C[W]^G) .
\end{tikzcd}
\end{center}

We refer to $X$ as the GIT stack quotient, $\mathfrak{X}$ as the stack quotient, $\underline{X}$ as the underlying coarse space or GIT quotient with respect to $\theta$ and $\underline{X}_0$ as the affine quotient. 

\begin{assumption}
We assume $W^{ss}(\theta)=W^s(\theta)$ so that $X$ is a quasi-compact Deligne-Mumford stack.
\end{assumption}

\begin{definition}\label{Definition: Line and vector bundles}
Fix a representation $\tau: G \to \GL(r,\C)$ for some integer $r$. Define the vector bundle $\mathcal{E} = \mathcal{E}_\tau \to\mathfrak{X}$ as
\begin{equation*}
    \mathcal{E} := [W\times\C^r/G] ,
\end{equation*}
where $g\in G$ acts on the $\C^r$ factor by multiplication by $\tau(g)$. For simplicity we will omit $\tau$ from the notation for the representation $\C^r$ and the vector bundle $\mathcal{E}$ throughout the paper.

The vector bundle $\mathcal{E}$ restricts to a vector bundle on $X$, which we denote by $E$.  It  may be realized as the quotient
$E:=[W^{ss}(\theta)\times \C^r/G].$
\end{definition}

\begin{definition}
Let $(x,g)$ be a point in a twisted sector $\mathfrak{X}_g$.  The fiber $\mathcal{E}_x$ splits as the direct sum of eigenspaces
\begin{equation*}
    \mathcal{E}_x = \bigoplus_{0\leq f<1} (\mathcal{E}_x)_f ,
\end{equation*}
where $g$ acts on $(\mathcal{E}_x)_f$ by multiplication by $e^{2\pi i f}$. Define the \emph{age} of $\mathcal{E}$ at $g$ as
\begin{equation*}
    \age_g (\mathcal{E}) := \sum_{0\leq f < 1} f \dim_\C (\mathcal{E}_x)_f  .
\end{equation*}
\end{definition}

\subsubsection{A local complete intersection in $X$}
\label{Section: A local complete intersections in X}

Let $\mathcal{E}$ and $E$ be as defined in Definition~\ref{Definition: Line and vector bundles}.
We do not assume $E$ is pulled back from the coarse space $\underline{X}$, hence the isotropy groups $G_x$ for $x$ in $X$ may act nontrivially on the fibers of $E$.

Let $s$ be a $G$-invariant global section of $\Gamma(W,W\times \C^r)^G$ such that the zero locus $Z(s)$ is an irreducible complete intersection intersecting the semistable locus $W^{ss}(\theta)$ non-trivially.  Assume $Z(s)\cap W^{ss}(\theta)$ is non singular.

\begin{definition}
Define $\mathfrak{Z}:=[Z(s)/G]$ to be the a closed substack of $\mathfrak{X}$ given by the section $s$.  Define $Z$ as the GIT stack quotient
\begin{equation*}
    Z := [(Z(s)\cap W^{ss}(\theta))/G].
\end{equation*}
\end{definition}
The semistable locus of $Z(s)$ is exactly the intersection of $Z(s)$ and $W^{ss}(\theta)$.  The section $s$ restricts to a section of $W^{ss}(\theta)\times\C^r$ whose zero locus is exactly $Z(s)\cap W^{ss}(\theta)$.  Equivalently, $\mathfrak Z$ and $Z$ are defined as the zero loci of the sections of $\mathcal{E}$ and $E$ induced by $s$.  We will also sometimes denote these sections by $s$.

\begin{remark}
If we assume the representation $\C^r$ splits as the direct sum of one-dimensional representations $\oplus_{i=1}^r\C_{\tau_i}$ for characters $\tau_i:G\to\C^*$, then $\mathcal{E}$ and $E$ split and $Z$ defines a complete intersection of $X$. 
\end{remark}

\begin{definition}
For a fixed $g$ in $S$, denote
\begin{equation*}
    E_g := [ (W^{ss}(\theta) \times \C^r )^g / C(g) ] .
\end{equation*}
\end{definition}
The inertia stack $IE$ may be written as $IE = \bigsqcup_{g\in S} E_g$.  Note that $E_g$ is not usually equal to $E|_{X_g}$.  When $g$ acts nontrivially on the representation $\C^r$, the twisted sector $E_g$ equals $X_g$.  If $g$ acts trivially, then $E_g$ is a vector bundle over $X_g$.  

The section $s$ induces a section $s_g\in\Gamma(X_g,E_g)$ for each $g\in S$.  The twisted sector $Z_g\subset IZ$ may be realized as the vanishing locus of $s_g$.  Denote by $j:IZ\to IX$ and $0_E:IX\to IE$ the inclusion and zero section respectively.  Then we have a fiber square
\begin{equation}
\label{Diagram: The bottom fiber square for the fiber cube}
\begin{tikzcd}
      IZ \arrow[r, "j"] \arrow[d, "j"]
    & IX  \arrow[d, "s"] \\
      IX \arrow[r, "0_{E}"]
    & IE .
\end{tikzcd}
\end{equation}

By abuse of notation, we will also denote by $j$, $s$, and $0_E$ the corresponding maps between the rigidified inertia stacks.

\begin{definition}
The image of the cohomology of $X_g$ under the pullback via the inclusion $j_g:Z_g\to X_g$ is called the \emph{ambient cohomology} of $Z_g$ and denoted $H^*_{\amb}(Z_g)$.  The \emph{ambient Chen-Ruan cohomology} of $Z$ is 
\begin{equation*}
    H^*_{\CR,\amb}(Z) := H^*_{\amb}(IZ) = \bigoplus_{g\in S} H^*_{\amb}(Z_g) .
\end{equation*}
\end{definition}

Restricting the computation of quasimap invariants of $Z$ to the ambient cohomology is common and, as Proposition~\ref{Proposition: Pullback of L^X/Z is equivalent to L^Z of the pullback} will suggests, natural.

\begin{assumption}
\label{Assumption: The Poincare pairing on the ambient cohomology of Z is non-degenerate}
Following \cite{iritani2016quantum} and \cite{Shoemaker_2018}, we assume the Poincar\'e pairing on $H^*_{\CR,\amb}(Z)$ is non-degenerate.  This is equivalent to assuming the cohomology of $Z$ splits as the direct sum:
\begin{equation*}
    H^*_{\CR} (Z) = \im(j^*) \oplus \ker(j_*).
\end{equation*}
\end{assumption}

\subsubsection{The total space $E^\vee$}


With the setting as in \S \ref{Section: A local complete intersections in X}, consider the vector bundles
\begin{equation*}
    \mathcal{E}^\vee := [W\times\C^r/G] \hspace{1 cm} E^\vee := [(W^{ss}(\theta) \times\C^r)/G] .
\end{equation*}
Note that $IE^\vee = \bigsqcup_{g\in S} E_g^\vee$.  In this section we describe the Chen--Ruan cohomology of compact type of $E^\vee$.  A more detailed introduction to cohomology of compact type appears in \S 2 of \cite{Shoemaker_2018}.  Note that in \cite{Shoemaker_2018}, the cohomology of compact type is referred to as the narrow cohomology.

\begin{definition}\parencite[Definition 2.1]{Shoemaker_2018} \label{ctpairing}
The \emph{cohomology of compact type} of $E^\vee_g$ is the image of the natural homomorphism 
\begin{equation*}
    \phi:H^*_{\cs}(E^\vee_g)\to H^*(E^\vee_g)
\end{equation*}
from compactly supported cohomology to cohomology. The \emph{Chen-Ruan cohomology of compact type} is
\begin{equation*}
    H^*_{\CR,\ct}(E^\vee) := H^*_{\ct}(IE^\vee) = \bigoplus_{g\in S} H^*_{\ct}(E^\vee_g) .
\end{equation*}
\end{definition}

Given a class $\gamma \in H^*_{\CR,\ct}(E^\vee)$, we call $\overline \gamma \in H^*_{\CR,\cs}(E^\vee)$ a \emph{lift} of $\gamma$ if $\phi(\overline \gamma) = \gamma$.  It is proven in \cite[Lemma 2.6]{Shoemaker_2018} that for $\gamma \in H^*_{\CR,\ct}(E^\vee)$ and $\kappa \in \ker(\phi) \subset H^*_{\CR,\cs}(E^\vee)$, $\kappa \cup \gamma$ is zero.  From this fact one can check that the following pairing is well-defined and nondegenerate.

\begin{definition}
Define the \emph{compact type pairing} on $E^\vee$ as follows.  For $\alpha, \beta \in H^*_{\CR,\ct}(E^\vee)$,
\begin{equation*}
    \langle \alpha, \beta\rangle^{E^\vee, \ct} := \langle \alpha, \overline \beta\rangle^{E^\vee} = \int_{IE^\vee} \alpha \cup \iota^* (\overline \beta),
\end{equation*}
where $\overline \beta \in H^*_{\CR,\cs}(E^\vee)$ is a lift of $\beta$.
\end{definition}

Let $p_g: E_g^\vee \to X_g$ denote the vector bundle projection for each $g\in S$.  Define the linear map
\begin{equation*}
    e(p_g^*E^\vee_g)\cup-:H^*(E^\vee_g)\to H^*(E^\vee_g) 
\end{equation*}
by cupping with the Euler class $e(p_g^*E^\vee_g)$.  Denote by $i:IX\hookrightarrow IE^\vee$ the inclusion induced by the zero section of $E^\vee$ and $i_g:X_g\hookrightarrow E_g^\vee$ the inclusion induced by the zero section of $E_g^\vee$.

\begin{lemma}\emph{\parencite[Proposition~2.15]{Shoemaker_2018}}
\label{Lemma: Narrow cohomology isomorphism}
If $X$ is compact, then the following groups are equal:
\begin{equation}
\label{Equation: Narrow cohomology}
    H^*_{\CR,\ct}(E^\vee) = \im\left( i_*:H^*_{\CR}(X)\to H^*_{\CR}(E^\vee) \right) = \bigoplus_{g\in S} \im\left( e(p^*E^\vee_g) \cup - \right) .
\end{equation}
\end{lemma}
\begin{proof}
We will prove the result for a given twisted sector $E^\vee_g$.  If $g$ fixes only the origin of $W\times\C^r$, then the equality holds trivially since $X$ is compact.  Let $i_{g*}^{\cs}$ denote the pushforward on compactly supported cohomology induced by $i_g$.  Note that $i_{g*}$ factors as
\begin{equation*}
    i_{g*} = i_{g*} \circ \phi = \phi \circ i_{g*}^{\cs} .
\end{equation*}
Thus, by \parencite[Equation~(6.11)]{bott&tu}, $i_{g*}^{\cs}$ is an isomorphism.  We obtain the first equality of \eqref{Equation: Narrow cohomology}:
\begin{equation*}
    \im\left( i_{g*}:H^*(X_g)\to H^*(E^\vee_g) \right) = \im\left( \phi \circ i_{g*}^{\cs} \right) = H^*_{\ct}(E^\vee_g) .
\end{equation*}

Now, let $\alpha$ be an element of $H^*(E_g^\vee)$.  Then
\begin{align*}
    e(p_g^*E^\vee_g) \cup \alpha &= i_{g*}(1) \cup \alpha \\
    &= i_{g*}( 1 \cup i_g^*(\alpha) ) \\
    &= i_{g*} ( i_g^* (\alpha) ) .
\end{align*}
The first equality holds because $X_g$ can be viewed as the zero locus of the tautological section of $p_g^*E_g^\vee$ and the second is the projection formula. 

Since the pullback via $i_g$ is an isomorphism, the second equality of \eqref{Equation: Narrow cohomology} holds.
\end{proof}

Later we will use Lemma~\ref{Lemma: Narrow cohomology isomorphism} to express classes in the Chen-Ruan cohomology of compact type of $E^\vee$ in terms of classes in the Chen-Ruan cohomology of $X$.  For $\gamma$ an element of $H^*_{\ct}(E^\vee_g)$, there exist a class $\alpha$ in $H^*(X_g)$ such that
\begin{equation}
\label{Equation: Rewriting an element of narrow cohomology}
        \gamma = i_{g*}(\alpha) = p_g^*( e(E^\vee_g) \cup \alpha ) .
\end{equation}

\subsection{Quasimaps to a GIT stack quotient}
\label{Section: Quasimaps to a GIT stack quotient}

In this section we recall the definition of $\varepsilon$-stable quasimap invariants and define the generating functions that will appear in our later theorems. Let $X$ be a GIT stack quotient as in~\S\ref{Section: GIT stack quotients}.

\begin{definition}
\parencite[\S 4]{Abramovich_Graber_Vistoli_Gromov_Witten_theory_of_Deligne_Mumford_stacks_2008}
Over an algebraically closed field, a \emph{twisted curve} is a connected, one-dimensional Deligne-Mumford stack which is \'etale locally a nodal curve, and a scheme outside the marked points and the singular locus.
\end{definition}

\begin{definition}
\parencite[Definition 2.1]{Cheong-Ciocan-Fontanine-Kim_2015}
Let $(C, x_1,\ldots, x_k)$ be a $k$-pointed, genus-$g$ twisted curve and let $\varphi: (C,x_1,\ldots,x_k)\to (\underline{C}, \underline{x}_1,\ldots,\underline{x}_k)$ be its rigidification.
\begin{itemize}
    \item A $k$-pointed, genus-$g$ \emph{quasimap to $X$} is a twisted curve $(C,x_1,\ldots,x_k)$ together with a representable morphism $[u]:C\to\mathfrak{X}$ such that $[u]^{-1}(\mathfrak{X}\setminus X)$, referred to as the \emph{base locus}, is purely zero-dimensional. We denote this quasimap as $(C,x_1,\ldots,x_k,[u])$ and call $(C,x_1,\ldots,x_k)$ or $C$ its \emph{source curve}.
    \item The \emph{class} $\beta$ of the quasimap is an element of $\Hom(\Pic\mathfrak{X},\Q)$ defined by
    \[ \beta: \Pic\mathfrak{X} \to \Q, \qquad \mathcal{L} \mapsto \deg\left( [u]^*(\mathcal{L}) \right) . \]
    The \emph{degree} of the quasimap $[u]$ is given by the rational number $\beta(\mathcal{L}_\theta)$. We say $\beta$ in $\Hom(\Pic\mathfrak{X},\Q)$ is \emph{$\theta$-effective} if it is represented by a quasimap to $X$. The set of $\theta$-effective classes forms a sub-semigroup of $\Hom(\Pic\mathfrak{X},\Q)$ denoted $\Eff(W,G,\theta)$. For brevity, we often denote this sub-semigroup as $\Eff$.
    \item If the base locus of a quasimap $(C,x_1,\ldots,x_k,[u])$ does not contain marked or nodal gerbes, then we call the quasimap \emph{prestable}.
    \item Let $\mathbf{e}$ be the least common multiple of $|\Aut(p)|$ for all geometric points $p\to X$ with isotropy groups $\Aut(p)$. Fix a rational number~$\varepsilon$. A prestable quasimap is called \emph{$\varepsilon$-stable} if 
    \begin{enumerate}
        \item the $\Q$-line bundle
        \begin{equation}\label{Equation: Quasimap stability condition}
            \omega_{\underline{C}} \left( \sum_{i=1}^k \underline{x}_i \right) \otimes \left( \varphi_*( [u]^*\mathcal{L}_\theta^{\otimes \mathbf{e}} ) \right)^{\varepsilon} 
        \end{equation}
        on the coarse curve $\underline{C}$ is ample and
        \item for all $x\in C$, 
        \begin{equation*}
            \varepsilon l(x) \leq 1,
        \end{equation*}
        where $l(x)$ is the \emph{length at $x$} defined in \parencite[\S7.1]{Ciocan_Fontanine_Kim_Maulik_Stable_quasimaps_to_GIT_quotients_2014}.
    \end{enumerate}
    We say a prestable quasimap is \emph{$0+$-stable}, if \eqref{Equation: Quasimap stability condition} is ample for all rational $\varepsilon>0$.
\end{itemize}
\end{definition}

\begin{definition}\label{Definition: quasimap moduli space}
The \emph{moduli space of $k$-pointed, genus-$g$, $\varepsilon$-stable quasimaps to $X$ of degree $\beta$}, denoted by $Q^{\varepsilon}_{g,k}(X,\beta)$, is the space of isomorphism classes of $k$-pointed, genus-$g$, $\varepsilon$-stable quasimaps to $X$ of degree $\beta$.
\end{definition}

By Theorem~2.7 of \cite{Cheong-Ciocan-Fontanine-Kim_2015}, $Q^{\varepsilon}_{g,k}(X,\beta)$ is a proper Deligne-Mumford stack over the affine quotient $\underline{X}_0$. Furthermore, when the singularities of $W$ are at worst local complete intersections and $W^{ss}(\theta)$ is nonsingular, $Q^{\varepsilon}_{g,k}(X,\beta)$ carries a canonical perfect obstruction theory.

When $\varepsilon$ is sufficiently large, the moduli space of $\varepsilon$-stable quasimaps $Q_{g,k}^{\varepsilon}(X,\beta)$ coincides with the moduli space of stable maps $\overline{\mathcal{M}}_{g,n}(X,\beta)$ from Gromov-Witten theory. We adopt the notation of \cite{Ciocan_Fontanine_Kim_Maulik_Stable_quasimaps_to_GIT_quotients_2014,Cheong-Ciocan-Fontanine-Kim_2015} and define 
\begin{equation*}
    Q_{g,k}^{\infty}(X,\beta) := \overline{\mathcal{M}}_{g,n}(X,\beta) .    
\end{equation*}

\subsection{The Dubrovin connection}
\label{Section: The Dubrovin connection}

Since the base locus of a $\varepsilon$-stable quasimap is disjoint from the marked gerbes for $i=1,\ldots, k$, there exists evaluation maps to the rigidified inertia stack $\bar IX$: 
\begin{equation*}
    ev_i : Q^{\varepsilon}_{g,k}(X,\beta) \to \bar I X .
\end{equation*}
We refer the reader to \parencite[\S 4.4]{Abramovich_Graber_Vistoli_Gromov_Witten_theory_of_Deligne_Mumford_stacks_2008} for an introduction to such maps. There is a canonical map $\bar\omega: IX \to \bar IX$ inducing an isomorphism of cohomology $\bar\omega_*: H^*_{\CR}(X) \to H^*(\bar IX)$. By composing $ev_i^*$ with $\bar\omega_*^{-1}$, we may define quasimap invariants using cohomology class insertions from $H^*_{\CR}(X)$.

Let $\psi_i$ represent the first Chern class of the universal cotangent line over $Q^{\varepsilon}_{g,k}(X,\beta)$ whose fiber over a point $(C,x_1,\ldots,x_k,[u])$ is given by the cotangent space of the underlying coarse curve $\underline{C}$ at the $i$th marked point. Denote by $[Q^{\varepsilon}_{g,k}(X,\beta)]^{\vir}$ the virtual fundamental class from \parencite[\S 4.5]{Ciocan_Fontanine_Kim_Maulik_Stable_quasimaps_to_GIT_quotients_2014} and \parencite[\S 2.4.5]{Cheong-Ciocan-Fontanine-Kim_2015}.  By \parencite[Theorem~2.7]{Cheong-Ciocan-Fontanine-Kim_2015}, $Q^{\varepsilon}_{g,k}(X,\beta)$ is proper over $\underline{X}_0$. By \parencite[Lemma~01W6]{stacks-project}, this implies each evaluation map $ev_i$ is proper.

\begin{definition}\label{dinvt}
Assume $X$ is proper. For non-negative integers $a_i$ and classes $\alpha_i\in H^*_{\CR}(X)$, \emph{$\varepsilon$-quasimap invariants} or simply \emph{quasimap invariants} are given by integrals 
\begin{equation} \label{Equation: Quasimap invariants}
    \left\langle \alpha_1 \psi_1^{a_1} , \ldots , \alpha_k \psi_k^{a_k} \right\rangle_{g,\beta}^{X,\varepsilon} := \int_{\left[Q^{\varepsilon}_{g,k}(X,\beta)\right]^{\vir}} \prod_{i=1}^k ev_i^*(\alpha_i) \psi_i^{a_i} .
\end{equation}
Fix a basis $\{T^i\}_{i\in I}$ of $H^*_{\CR}(X)$ and let $\boldsymbol{t}=\sum_{i\in I} t_i T^i$.  Define the double-bracket
\begin{equation}\label{equation: double brackets}
    \left\langle\left\langle \alpha_1 \psi_1^{a_1} , \ldots , \alpha_k \psi_k^{a_k} \right\rangle\right\rangle^{X,\varepsilon}_{0} (\boldsymbol{t}) := \sum_{\beta \in \Eff} \sum_{m \geq 0} \frac{q^\beta}{m!} \left\langle \alpha_1 \psi_1^{a_1} , \ldots , \alpha_k \psi_k^{a_k}, \boldsymbol{t}, \ldots, \boldsymbol{t} \right\rangle_{0,\beta}^{X,\varepsilon}  .
\end{equation}
\end{definition}

If $X$ is not proper (i.e. if $\underline{X}_0$ is not a point), then slightly more care must be taken to define quasimap invariants.  If at least one class $\alpha_j$ lies in compactly supported cohomology $H^*_{\CR, \cs}(X)$ for $1 \leq j \leq k$ then we define $$\left\langle \alpha_1 \psi_1^{a_1} , \ldots , \alpha_k \psi_k^{a_k} \right\rangle_{g,\beta}^{X,\varepsilon}$$ exactly as in \eqref{Equation: Quasimap invariants}.

\begin{definition}\label{nonproper brackets}
We also define $\varepsilon$-quasimap invariants whenever at least two insertion classes are compact type. For simplicity assume $\alpha_j$ and $\alpha_k$ are in $H^*_{\CR, \ct}(X)$ for $ j \leq k$.  Define
\begin{equation} \label{Equation: nonproper invariants}
\left\langle \alpha_1 \psi_1^{a_1} , \ldots , \alpha_k \psi_k^{a_k} \right\rangle_{g,\beta}^{X,\varepsilon} :=
 \left\langle \tilde{ev}_{k*} \left( \prod_{i=1}^{k-1} ev_i^*(\alpha_i) \psi_i^{a_i} \cup \psi_k^{a_k} \cap \left[Q^{\varepsilon}_{0,k}(X,\beta)\right]^{\vir} \right), \alpha_k  \right\rangle^{X ,\ct},
\end{equation}
where $\langle -  , - \rangle^{X, \ct}$ is the compact type pairing of Definition~\ref{ctpairing} and $\tilde{ev}_{k} = \nolinebreak \iota \circ \nolinebreak ev_k$.  Define the double bracket exactly as in \eqref{equation: double brackets}, but replacing \eqref{Equation: Quasimap invariants} with~\eqref{Equation: nonproper invariants}.
\end{definition}
By \cite[Proposition 2.5]{Shoemaker_2018}, pullback and pushforward via a proper map each preserve the compact type subspace  of $ H^*_{\CR}(X)$, therefore the pushforward in the right hand side of \eqref{Equation: nonproper invariants} lies in $H^*_{\CR, \ct}(X)$ as desired.

Let $\{T_i\}_{i\in I} \subseteq H^*_{\CR,\cs}(X)$ denote the dual basis to $\{T^i\}_{i\in I}$ in compactly supported cohomology.  Whether or not $X$ is proper, the expression $$\tilde{ev}_{k*} \left( \prod_{i=1}^{k-1} ev_i^*(\alpha_i) \psi_i^{a_i} \cup \psi_k^{a_k} \cap \left[Q^{\varepsilon}_{0,k}(X,\beta)\right]^{\vir} \right)$$ 
is equal to
$$\sum_{i\in I} \left\langle \alpha_1 \psi_1^{a_1} , \ldots , T_i \psi_k^{a_k} \right\rangle_{g,\beta}^{X,\varepsilon} T^i.$$

\begin{definition} \parencite[\S 2]{Ciocan-Fontanine_Kim_Higher_genus_wall-crossing_2017}
The \emph{$\varepsilon$-quantum product} of $\alpha_1$ and $\alpha_2$, written $\alpha_1  \bullet^{X,\varepsilon}_{\boldsymbol{t}} \nolinebreak \alpha_2 $, is the sum
\begin{equation*}
    \alpha_1 \bullet^{X,\varepsilon}_{\boldsymbol{t}} \alpha_2 := \sum_{i\in I} \langle\langle \alpha_1, \alpha_2, T_i \rangle\rangle^{X,\varepsilon}_{0} (\boldsymbol{t}) T^i .
\end{equation*}
\end{definition}

\begin{definition}
Let $z$ be a formal variable. The \emph{$\varepsilon$-Dubrovin connection} on $H^*_{\CR}(X)\otimes \C[[t_i,q]]_{i\in I}[z,z^{-1}]$ is defined by 
\begin{equation}
\label{Equation: Dubrovin Connection}
    \nabla_i^{X,\varepsilon} = \frac{\partial}{\partial t_i} + \frac{1}{z} T_i \bullet^{X,\varepsilon}_{\boldsymbol{t}} .
\end{equation}
\end{definition}

 Define the operator $L^{X,\varepsilon}(\boldsymbol{t},z)$ by
\begin{equation}
\label{Equation: General Fundamental Solution}
    L^{X,\varepsilon}(\boldsymbol{t},z)(\alpha) := \alpha + \sum_{i\in I} \left\langle\left\langle \frac{\alpha}{-z-\psi} , T_i \right\rangle\right\rangle^{X,\varepsilon}_0 (\boldsymbol{t}) T^i ,
\end{equation}
for all $\alpha\in H^*_{\CR}(X)$.

\begin{proposition}
\label{Proposition: Fundamental solution to the quantum differential equation}
The Dubrovin connection $\nabla^{X,\varepsilon}$ is flat, with fundamental solution given by 
the operator $L^{X,\varepsilon}(\boldsymbol{t},z)$. For $i\in I$ and $\alpha\in H^*_{\CR}(X)$ we have the equality
\begin{equation*}
    \nabla^{X,\varepsilon}_i\left( L^{X,\varepsilon}(\boldsymbol{t},z)(\alpha) \right) = 0 .
\end{equation*}
\end{proposition}

\begin{proof}
The proof is identical to the Gromov-Witten theory case in \cite{Cox_Katz_Mirror_symmetry_and_algebraic_geometry_1999,Iritani-Quantum_Cohomology_and_periods_2011}. The key ingredient is the topological recursion relation for quasimaps, which appears in \parencite[Corollary~2.3.4]{Ciocan-Fontanine_Kim_Higher_genus_wall-crossing_2017}.
\end{proof}

Denote by $L^{X,\varepsilon}(z)$ the restriction of $L^{X,\varepsilon}(\boldsymbol{t},z)$ to $\boldsymbol{t}=\boldsymbol{0}$:
\begin{align}
\label{Equation: Fundamental Solution}
    L^{X,\varepsilon}(z)(\alpha) := L^{X,\varepsilon}(\boldsymbol{0},z)(\alpha) &= \alpha + \sum_{i\in I} \sum_{\beta\in\Eff} q^\beta \left\langle \frac{\alpha}{-z-\psi} , T_i \right\rangle^{X,\varepsilon}_{0,\beta} T^i  \\
    &= \alpha + \sum_{\beta\in \Eff} q^\beta \tilde{ev}_{2*} \left( \frac{ev_1^*(\alpha)}{-z-\psi_1} \cap \left[ Q_{0,2}^{\varepsilon}(X,\beta) \right]^{\vir} \right) . \nonumber
\end{align}
The operator $L^{X,\varepsilon}(z)$ records all two-pointed genus-zero $\varepsilon$-stable quasimap invariants with one primary insertion and one descendants insertion. In this paper we restrict our attention to $L^{X,0+}(z)$.


\section{Two-pointed genus-zero quasimaps}
\label{Section: Two-pointed genus-zero quasimaps}
In this section we define weak forms of convexity and concavity and show that both are equivalent. We also prove that weakly semi-positive vector bundles $\mathcal{E} \to \mathfrak{X}$ (Definition~\ref{Definition: Weakly semi-positive}) are weakly convex.

\subsection{Source curves}
\label{Section: Source curves}

Here we describe the source curves of two-pointed genus-zero $0+$-stable quasimaps.
\begin{lemma}
\label{Lemma: Source curves are chains of smooth curves with exactly two special points}
For a point $(C,x_1,x_2,[u])$ in $Q^{0+}_{0,2}(X,\beta)$, the underlying coarse curve $(\underline{C},\underline{x}_1,\underline{x}_2)$ is an at worst nodal curve such that each irreducible component has exactly two special points. Furthermore, the degree of $[u]$ is positive on every rational component.
\end{lemma}

\begin{proof}
For a point $(C,x_1,x_2,[u])$ in $Q^{0+}_{0,2}(X,\beta)$, the dual graph of its source curve $(C,x_1,x_2)$ is a tree because the curve is genus zero. The stability condition \eqref{Equation: Quasimap stability condition} states that the line bundle
\begin{equation*}
    \omega_{\underline{C}} \left( \underline{x}_1 + \underline{x}_2 \right) \otimes \left( \phi_*( [u]^*\mathcal{L}_\theta^{\otimes \mathbf{e}} ) \right)^{\varepsilon} 
\end{equation*}
has positive degree on each rational component of the underlying coarse curve $(\underline{C},\underline{x}_1,\underline{x}_2)$ for all $\varepsilon>0$. Hence, every rational component of $(\underline{C},\underline{x}_1,\underline{x}_2)$ with a single node must contain a marked point. Since there are only two marked points, there are at most two rational components with a single node. That is, $(\underline{C},\underline{x}_1,\underline{x}_2)$ must be a chain of rational components with marked points on the terminal components. 

Therefore each rational component of $\underline{C}$ has exactly two special points. The stability condition then implies that the degree of $[u]$ must be positive on every component of $(C,x_1,x_2)$.
\end{proof}

\begin{figure}[h!]
    \centering
    \begin{tikzpicture}
    \draw (-5,0) -- (-3,0);
    \filldraw[black] (-4.5,0) circle (1pt) node {};
    \filldraw[black] (-3.5,0) circle (1pt) node {};

    \draw (-1.5,-.5) -- (-.5,.5);
    \draw (-1,.5) -- (0,-.5);
    \filldraw[black] (-1.25,-.25) circle (1pt) node {};
    \filldraw[black] (-.25,-.25) circle (1pt) node {};

    \draw (1.45,.25) -- (2.8,.25);
    \draw (1.75,-.5) -- (1.75,.5);
    \draw (2.5,-.5) -- (2.5,.5);
    \filldraw[black] (1.75,-.25) circle (1pt) node {};
    \filldraw[black] (2.5,-.25) circle (1pt) node {};

    \draw (3.75,0) node {\Large{$\cdots$}};

    \end{tikzpicture}
    \caption{The underlying coarse curves $(\underline{C},\underline{x}_1,\underline{x}_2)$.}
    \label{Figure: source curves}
\end{figure}

For integers $c,d$, denote by $\CP_{[c,d]}$ a smooth twisted curve whose coarse space is $\CP^1$ with two marked points such that the isotropy groups at the marked points are $\mu_c$ and $\mu_d$. By Lemma~\ref{Lemma: Source curves are chains of smooth curves with exactly two special points}, each component of a source curve in $Q_{0,2}^{0+}(X,\beta)$ is given by some $\CP_{[c,d]}$. 

Denote by $l$, $a$, and $b$ the integers $l=\gcd(c,d)$, $a=c/l$, and $b=d/l$. We can rewrite $l$ as the product of integers $l_1$ and $l_2$ such that $\gcd(l_1,l_2)=1$, $\gcd(l_1,b)=1$, and $\gcd(l_2,a)=1$. In the next lemma we give a uniform way of expressing $\CP_{[c,d]}$ as a quotient stack.

\begin{lemma}\label{Lemma: two-pointed smooth twisted curves are isomorphic to a GIT quotient}
Let $\chi:\mu_{l_1} \times \mu_{l_2} \times \C^* \to \C^*$ be the character given by 
\begin{equation*}
    \chi\left( e^{2\pi i \frac{m_1}{l_1}},e^{2\pi i \frac{m_2}{l_2}},\lambda \right)=\lambda .
\end{equation*}
The twisted curve $\CP_{[c,d]}$ may be represented as a GIT stack quotient
\begin{equation}
\label{Equation: Irreducible Component}
    \left[ \C^2 \sslash_{\hspace{-2pt}\chi} (\mu_{l_1} \times \mu_{l_2} \times \C^*) \right] 
\end{equation}
with action 
\begin{equation*}
    \left( e^{2\pi i \frac{m_1}{l_1}},e^{2\pi i \frac{m_2}{l_2}},\lambda \right) \cdot \left( x,y \right) = \left( \lambda^a x e^{2\pi i \frac{m_1}{l_1}}, \lambda^b y e^{2\pi i \frac{m_2}{l_2}} \right) ,
\end{equation*}
for $0\leq m_1 < l_1$, $0\leq m_2 < l_2$, $\lambda\in\C^*$, and $(x,y)\in\C^2$.
\end{lemma}

\begin{proof}
It suffices to show that the orbifold curve \eqref{Equation: Irreducible Component} contains exactly two orbifold points with isotropy groups of orders $c$ and $d$. 

We first calculate the generic isotropy. Fix $x,y\neq0$ and let $\left(e^{2\pi i \frac{m_1}{l_1}}, e^{2\pi i \frac{m_2}{l_2}}, \lambda\right)$ be an element of $\mu_{l_1}\times \mu_{l_2}\times \C^*$ which fixes $(x,y)$. Then we have
\begin{equation}
    \label{Equation: proof of two-pointed smooth twisted curves are isomorphic to a GIT quotient - equation 1}
    \lambda^a e^{2\pi i \frac{m_1}{l_1}} = 1    
\end{equation}
and
\begin{equation}
    \label{Equation: proof of two-pointed smooth twisted curves are isomorphic to a GIT quotient - equation 2}
    \lambda^b e^{2\pi i \frac{m_2}{l_2}} = 1 .    
\end{equation}
Equality \eqref{Equation: proof of two-pointed smooth twisted curves are isomorphic to a GIT quotient - equation 1} implies $\lambda=e^{2\pi i \frac{l_1n-m_1}{al_1}}$ for some integer $n$. So, the product $\lambda^b e^{2\pi i \frac{m_2}{l_2}}$ equals $e^{2\pi i \left( \frac{m_2}{l_2} + \frac{bl_1n-bm_1}{al_1} \right) }$ and \eqref{Equation: proof of two-pointed smooth twisted curves are isomorphic to a GIT quotient - equation 2} implies
\begin{equation*}
    \frac{m_2}{l_2} + \frac{bl_1n-bm_1}{al_1} = \frac{al_1m_2+bl_1l_2n-bl_2m_1}{al_1l_2}
\end{equation*}
is an integer. 

If $al_1l_2$ divides $al_1m_2+bl_1l_2n-bl_2m_1$, then $l_2$ divides $al_1m_2$. However, the greatest common divisor of $al_1$ and $l_2$ is one, and $m_2$ is a non-negative integer strictly less than $l_2$. Thus, $m_2$ must equal $0$. Similar reasoning shows $m_1$ also equals $0$.

In this case, \eqref{Equation: proof of two-pointed smooth twisted curves are isomorphic to a GIT quotient - equation 1} and \eqref{Equation: proof of two-pointed smooth twisted curves are isomorphic to a GIT quotient - equation 2} reduce to the equation $\lambda^a=\lambda^b=1$. By assumption, the greatest common divisor of $a$ and $b$ is one. Therefore \eqref{Equation: proof of two-pointed smooth twisted curves are isomorphic to a GIT quotient - equation 1} holds if and only if $\left( e^{2\pi i \frac{m_1}{l_1}},e^{2\pi i \frac{m_2}{l_2}},\lambda \right)$ is the identity.

An element $\left( e^{2\pi i \frac{m_1}{l_1}},e^{2\pi i \frac{m_2}{l_2}},\lambda \right)$ in $\mu_{l_1} \times \mu_{l_2} \times \C^*$ fixes $(1,0)$ if and only if $\lambda^a e^{2\pi i \frac{m_1}{l_1}}=1$. Hence, $\mu_{l_2}$ fixes $(1,0)$ and the subgroup of $\mu_{l_1}\times\C^*$ which fixes $(1,0)$ is
\begin{equation}
\label{e: subgroup which fixes (1,0)}
    \left\langle \left( e^{-2\pi i \frac{1}{l_1}} , e^{2\pi i \frac{1}{al_1}} \right) \right\rangle \cong \mu_{al_1} .
\end{equation}
This gives us the following isotropy group:
\begin{equation*}
    G_{(1,0)} \cong \mu_{al_1} \times \mu_{l_2} \cong \mu_{al_1l_2} = \mu_{c} .
\end{equation*}
The second isomorphism follows from the fact that $\gcd(al_1,l_2)=1$. The final equality results from the equalities $l=l_1l_2$ and $a=c/l$.

An identical argument shows that the order of the isotropy group at $(0,1)$ is $d$. One easily checks that the $\chi$-stable locus is $\C^2\setminus\{(0,0)\}$ as desired.
\end{proof}

The GIT stack quotient $\left[ \C^2 \sslash_{\hspace{-2pt}\chi} (\mu_{l_1} \times \mu_{l_2} \times \C^*) \right]$ may be written as a global quotient of a weighted projective space by a finite cyclic group. Recall that $l=l_1l_2$ and $\gcd(l_1,l_2)=1$. We will sometimes write $\left[ \C^2 \sslash_{\hspace{-2pt}\chi} (\mu_{l_1} \times \mu_{l_2} \times \C^*) \right]$ simply as $\CP(a,b)/\mu_l$, where the action of $\mu_l=\mu_{l_1}\times\mu_{l_2}$ is understood to be as described.

\begin{corollary}\label{Lemma: Orbifold Source Curves}
Let $(C,x_1,x_2)$ be a source curve of a two-pointed genus-zero $0+$-stable quasimap to a GIT stack quotient. Then the rational components of $(C,x_1,x_2)$ are isomorphic to a quotient of a weighted projective $\CP(a,b)/\mu_l$ as above.
\end{corollary}

\begin{proof}
The result follows immediately from Lemmas~\ref{Lemma: Source curves are chains of smooth curves with exactly two special points} and \ref{Lemma: two-pointed smooth twisted curves are isomorphic to a GIT quotient}.
\end{proof}

We conclude our discussion of two-pointed genus-zero $(0+)$-stable quasimap source curves with a remark about line bundles over such curves.

\begin{remark}
\label{Remark: Line bundles over the source curves}
Let $(C,x_1,x_2,[u])$ be a quasimap in the moduli space $Q_{0,2}^{0+}(X,\beta)$. By Lemma~\ref{Lemma: Source curves are chains of smooth curves with exactly two special points}, Lemma~\ref{Lemma: two-pointed smooth twisted curves are isomorphic to a GIT quotient}, and \parencite[Proposition~2.2]{martens2012variations}, line bundles over $C$ restricted to an irreducible component are isomorphic to a GIT stack quotient
\begin{equation*}
    \left[ \left( \C^2\setminus\{\overline{0}\} \times \C \right) \sslash_{\hspace{-2pt}\chi} \left( \mu_{l_1} \times \mu_{l_2} \times \C^* \right) \right]    
\end{equation*}
with the action given by
\begin{equation*}
    \left( e^{2\pi i \frac{m_1}{l_1}}, e^{2\pi i \frac{m_2}{l_2}}, \lambda \right) \cdot (x,y,z) = \left( \lambda^a x e^{2\pi i \frac{m_1}{l_1}}, \lambda^b y e^{2\pi i \frac{m_2}{l_2}}, \lambda^d z e^{2\pi i \frac{k_1 l_2 m_1 + k_2 l_1 m_2}{l_1l_2}} \right) ,
\end{equation*}
for some integers $k_1$, $k_2$, and $d$. Denote this GIT quotient by $\mathcal{O}_{\CP(a,b)/\mu_l}^{k_1,k_2}(d)$. We omit the superscripts $k_1$ and $k_2$ when $l$ equals $1$.
\end{remark}


\subsection{Weak convexity} 
\label{Section: Weak convexity}

We now define (weak) semi-positivity for a vector bundle over $\mathfrak{X}$ and prove that it implies a weak form for convexity.

\begin{definition}
\label{Definition: Semi-positive}
For a character $\tau:G\to \C^*$ with one-dimensional representation $\C_{\tau}$, denote by $\mathcal{L}=[W\times\C_\tau/G]$ the corresponding line bundle over $\mathfrak{X}$. We say $\mathcal{L}$ is \emph{positive} if $\beta(\mathcal{L}) > 0$ for all $\beta\in \Eff(W,G,\theta)$ and \emph{semi-positive} if $\beta(\mathcal{L})\geq 0$ for all $\beta\in \Eff(W,G,\theta)$.

A vector bundle $\mathcal{E}$ over $\mathfrak{X}$ is \emph{positive} (\emph{semi-positive}) if it splits as the direct sum of positive (semi-positive) line bundles.
\end{definition}

\begin{remark}
Our definition of a positive (semi-positive) line bundle $\mathcal{L}$ agrees with the definition in \parencite[\S 6.2]{Ciocan_Fontanine_Kim_Maulik_Stable_quasimaps_to_GIT_quotients_2014} of a positive (semi-positive) character~$\tau$.
\end{remark}

The authors of \cite{martens2012variations} generalize the Birkhoff-Grothendieck theorem to orbifolds whose coarse space is $\CP^1$ with only two points with nontrivial isotropy and chains of projective lines meeting at nodal singularities. These are exactly the source curves of $Q_{0,2}^{0+}(X,\beta)$ that we are considering.

Fix a quasimap $(C,x_1,x_2,[u]) \in Q_{0,2}^{0+}(X,\beta)$. By Lemma~\ref{Lemma: Source curves are chains of smooth curves with exactly two special points} and \cite{martens2012variations}, vector bundles over $C$ split as line bundles. Hence, for any vector bundle $\mathcal{E}\to\mathfrak{X}$ as in Definition~\ref{Definition: Line and vector bundles}, 
the pullback $[u]^*\mathcal{E}$ splits as the direct sum of line bundles $\oplus_{i=1}^r \mathcal{L}_i$ regardless of whether $\mathcal{E}$ splits.  This allows us to make the following definition:
\begin{definition}
\label{Definition: Weakly semi-positive}
The vector bundle $\mathcal{E}$ is \emph{weakly semi-positive} if, for any $\beta\in\Eff(W,G,\theta)$ and $(C,x_1,x_2,[u])\in Q_{0,2}^{0+}(X,\beta)$, the pullback $[u]^*\mathcal{E}$ splits as the direct sum of line bundles $\oplus_{i=1}^r \mathcal{L}_i$ such that 
$\deg(\mathcal{L}_i)\geq 0$ for all $1\leq i \leq r$.
\end{definition}

If $X$ is a smooth variety, convexity follows from semi-positivity of $E$.  This is no longer the case when $X$ is an orbifold (See Remark~\ref{r:conv} for more details).  In this section we consider a weaker notion, which we term weak convexity.  

\begin{definition}
\label{Definition: weak convexity}
A vector bundle $\mathcal{E}$ over $\mathfrak{X}$ is \emph{weakly convex} if $H^1(C,[u]^*\mathcal{E}(-x_2))$ vanishes for all $(C,x_1,x_2,[u])\in Q^{0+}_{0,2}(X,\beta)$ and $\beta\in \Eff(W,G,\theta)$.
\end{definition}

Assume the vector bundle $\mathcal{E}$ is weakly semi-positive. For a quasimap $(C,x_1,x_2,[u]) \in Q_{0,2}^{0+}(X,\beta)$, the group $H^1(C,[u]^*\mathcal{E})$ splits as
\begin{equation*}
    H^1(C,[u]^*\mathcal{E}) = \bigoplus_{i=1}^r H^1(C,\mathcal{L}_i) .
\end{equation*}

The line bundles $\mathcal{L}_i$ have non-negative degree by assumption.  Furthermore, each $\mathcal{L}_{i}$ has non-negative degree on every irreducible component of $C$.  Let $l$, $a$, and $b$ be as in \S \ref{Section: Source curves}. As noted in Remark~\ref{Remark: Line bundles over the source curves}, the restriction of $\mathcal{L}_{i}$ to an irreducible component of $C$ is then isomorphic to $\mathcal{O}_{\CP(a,b)/\mu_l}^{k_1,k_2}(d)$ for some integers $k_1$, $k_2$, and $d$ with $d\geq 0$.

We will show that $H^1(C,\mathcal{L})$ vanishes whenever $C$ is a two-pointed genus-zero $0+$-stable quasimap source curve and $\mathcal{L}\to C$ has non-negative degree on each irreducible component. We begin with the case that $C$ is smooth.

For Lemmas~\ref{Lemma: Convexity for P(a,b)/mu_l} and~\ref{Lemma: Global Sections of L and L|p}, let $(x,y)$ be the homogeneous coordinates of $\CP(a,b)/\mu_l$ and label $x_1=(1,0)$ and $x_2=(0,1)$. Also fix a non-negative integer $d$.

\begin{lemma}\label{Lemma: Convexity for P(a,b)/mu_l}
The cohomology group $H^1\left(\CP(a,b)/\mu_l,\mathcal{O}^{k_1,k_2}_{\CP(a,b)/\mu_l}(d)\right)$ vanishes.
\end{lemma}

\begin{proof}
First consider the case $\mu_l$ is trivial. Then $\CP(a,b)/\mu_l$ is a weighted projective space $\CP(a,b)$ and $\mathcal{O}^{k_1,k_2}_{\CP(a,b)/\mu_l}(d)$ is simply $\mathcal{O}_{\CP(a,b)}(d)$. Let $\varphi$ denote the rigidification map from $\CP(a,b)$ to $\CP^1$. 

The points $x_1$ and $x_2$ are the only points of $\CP(a,b)$ with nontrivial isotropy. The generator of the isotropy group at $L|_{x_j}$ acts by multiplication by $e^{2\pi i\frac{w_j}{ab}}$ for some weight $0\leq w_j < ab$. Orbifold Reimann-Roch gives the following:
\begin{align*}
    \deg( \varphi_* \mathcal{O}_{\CP(a,b)}(d)) &= \frac{d}{ab} - \left( \frac{w_1}{ab} + \frac{w_2}{ab} \right) \\
    &> \frac{d}{ab} - 2 \\ 
    &> -2 .
\end{align*}

Since the degree of $\varphi_* \mathcal{O}_{\CP(a,b)}(d)$ is an integer no less than $-1$, we conclude that 
\begin{equation}
\label{Equation: H^1 vanishes when l=1}
    H^1(\CP(a,b),\mathcal{O}_{\CP(a,b)}(d)) = H^1(\CP^1,\varphi_*\mathcal{O}_{\CP(a,b)}(d)) = 0 .
\end{equation} 

Now let $\mu_l$ be an arbitrary finite cyclic group. We have the fiber diagram:
\begin{equation*}
\begin{tikzcd}
    \mathcal{O}_{\CP(a,b)}(d) \arrow[d] \arrow[r]
    & \mathcal{O}^{k_1,k_2}_{\CP(a,b)/\mu_l}(d) \arrow[d] \\
    \CP(a,b) \arrow[r]
    & \CP(a,b)/\mu_l .
\end{tikzcd}
\end{equation*}

So, $H^1\left(\CP(a,b)/\mu_l, \mathcal{O}^{k_1,k_2}_{\CP(a,b)/\mu_l}(d)\right)$ is the $\mu_l$-invariant part of $H^1\left(\CP(a,b), \mathcal{O}_{\CP(a,b)}(d)\right)$. Equation~\eqref{Equation: H^1 vanishes when l=1} implies $H^1\left(\CP(a,b), \mathcal{O}_{\CP(a,b)}(d)\right)$ vanishes when $d\geq 0$. Therefore $H^1\left(\CP(a,b)/\mu_l, \mathcal{O}^{k_1,k_2}_{\CP(a,b)/\mu_l}(d)\right)$ vanishes.
\end{proof}

\begin{lemma}\label{Lemma: Global Sections of L and L|p}
If the isotropy at $x_j$ acts nontrivially on the fibers of $\mathcal{O}^{k_1,k_2}_{\CP(a,b)/\mu_l}(d)$, then $H^0\left(x_j,\mathcal{O}^{k_1,k_2}_{\CP(a,b)/\mu_l}(d)|_{x_j}\right)$ vanishes. Otherwise, there exists a section of $\mathcal{O}^{k_1,k_2}_{\CP(a,b)/\mu_l}(d)$ that is nonzero at $x_j$ and zero at the other point.
\end{lemma}

\begin{proof}
The first claim is immediate. For the second, assume the isotropy at $x_1$ acts trivially.

By \eqref{e: subgroup which fixes (1,0)}, the isotropy group at $x_1$ is
\begin{equation*} 
    G_{x_1} = \left\langle \left( e^{2\pi i \frac{1}{l_1}}, e^{2\pi i \frac{1}{l_2}}, e^{ -2\pi i \frac{1}{al_1} } \right) \right\rangle . 
\end{equation*}
The action of $\left( e^{2\pi i \frac{1}{l_1}}, e^{2\pi i \frac{1}{l_2}}, e^{ -2\pi i \frac{1}{al_1} } \right)$ on $\mathcal{O}^{k_1,k_2}_{\CP(a,b)/\mu_l}(d)|_{x_1}$ is given by
\begin{equation*}
    \left( e^{2\pi i \frac{1}{l_1}}, e^{2\pi i \frac{1}{l_2}}, e^{ -2\pi i \frac{1}{al_1} } \right) \cdot (1,0,z) = \left( 1, 0, z e^{ 2\pi i\frac{al_2k_1+al_1k_2-l_2d}{al_1l_2} }  \right) . 
\end{equation*}
This action is trivial, therefore $al_1l_2$ divides $al_2k_1+al_1k_2-l_2d$. Thus $a$ divides $l_2d$. Since $l_2$ and $a$ are coprime by the assumption preceding Lemma~\ref{Lemma: two-pointed smooth twisted curves are isomorphic to a GIT quotient}, this implies $a$ divides $d$. Rewrite $d$ as the product $ac$ for some integer $c$. Similar arguments show that $l_2$ divides $k_2$ and $l_1$ divides $(k_1-c)$. The latter ensures $k_1$ is congruent to $c$ modulo $l_1$.

Consider the map $f: \C^2 \to \C$ defined by $f(x,y)=x^c$. We claim $f$ descends to a section of $\mathcal{O}^{k_1,k_2}_{\CP(a,b)/\mu_l}(d)$. Note that
\begin{align*}
    f\left( \lambda^a x e^{2\pi i \frac{m_1}{l_1}}, \lambda^b y e^{2\pi i \frac{m_2}{l_2}} \right) &= \lambda^{ac} x^c e^{2\pi i \frac{cm_1}{l_1}} \\
    &= \lambda^{d} e^{2\pi i \frac{k_1 l_2 m_1 + k_2 l_1 m_2}{l_1l_2}}  f(x,y) .
\end{align*}
The second equality holds because $k_1$ is congruent to $c$ modulo $l_1$, $d$ equals $ac$, and $l_2$ divides $k_2$ so $\frac{k_2l_1m_2}{l_1l_2}$ is an integer for all integers $0\leq m_2 < l_2$. 

This verifies the map $(x,y) \mapsto (x,y,f(x,y))$ is $G$-equivariant, and therefore $f$ descends to a section $s$ of $\mathcal{O}^{k_1,k_2}_{\CP(a,b)/\mu_l}(d)$. One sees immediately that $s(x_1)\neq 0$ and $s(x_2)=0$.

An identical argument shows that if the action of the isotropy group at $x_2$ is trivial, then there exists a section $\tilde s$ of $\mathcal{O}^{k_1,k_2}_{\CP(a,b)/\mu_l}(d)$ such that $\tilde s(x_1) = 0$ and $\tilde s(x_2) \neq 0$.
\end{proof}

\begin{proposition}
\label{Proposition: Convexity for orbifold hypersurfaces}
Let $(C,x_1,x_2)$ be the source curve of a genus-zero $0+$-stable quasimap to a GIT stack quotient and $\mathcal{L}\to C$ a line bundle with non-negative degree on each irreducible component of $C$. Then $H^1(C,\mathcal{L}(-x_2))$ vanishes.
\end{proposition}

\begin{proof}
By Lemma~\ref{Lemma: Source curves are chains of smooth curves with exactly two special points}, the underlying coarse curve $\underline{C}$ is a chain of rational components such that the marked points $\underline{x}_1$ and $\underline{x}_2$ lie on the end components. Suppose $C$ has $k$ irreducible components $\{C_j\}_{j=1}^k$, labeled so that $x_1$ lies on $C_1$, $x_2$ lies on $C_k$, and the curves $C_j$ and $C_{j+1}$ intersect at the node $n_j$. 

We first show $H^1(C,\mathcal{L})$ vanishes. Consider the normalization sequence:
\begin{equation}
\label{Equation: Normalization exact sequence}
\begin{tikzcd}
    0 \arrow[r]
    &\mathcal{O}_C \arrow[r]
    &\displaystyle\bigoplus_{j=1}^k \mathcal{O}_{C_j} \arrow[r]
    &\displaystyle\bigoplus_{j=1}^{k-1} \mathcal{O}_C|_{n_j} \arrow[r]
    &0.
\end{tikzcd}
\end{equation}
Tensoring by $\mathcal{L}$ and taking cohomology, we obtain
\begin{equation*}
\begin{tikzcd}
    0 \arrow[r]
    & H^0(C,\mathcal{L}) \arrow[r] 
    & \displaystyle \bigoplus_{j=1}^k H^0(C_j,\mathcal{L}|_{C_j}) \arrow[r,"F"]
    & \displaystyle \bigoplus_{j=1}^{k-1} H^0(n_j,\mathcal{L}|_{n_j}) \arrow[r]
    & ~ \\
    ~ \arrow[r]
    & H^1(C,\mathcal{L}) \arrow[r] 
    & 0.
\end{tikzcd}
\end{equation*}
Here $\bigoplus_{j=1}^{k} H^1(C_j,\mathcal{L}|_{C_j})$ is zero by Corollary~\ref{Lemma: Orbifold Source Curves} and Lemma~\ref{Lemma: Convexity for P(a,b)/mu_l}. 

For sections $s_j$ in $H^0(C_j,\mathcal{L}|_{C_j})$, the map $F$ is defined as
\begin{equation*} 
    F(s_1,\ldots,s_k) = \left( s_1(n_1)-s_2(n_1), s_2(n_2) - s_3(n_2), \ldots, s_{k-1}(n_{k-1}) - s_k(n_{k-1}) \right). 
\end{equation*}
If the isotropy group $G_{n_j}$ acts nontrivially on the fiber $\mathcal{L}|_{n_j}$, then $H^0(n_j,\mathcal{L}|_{n_j})$ is zero. Otherwise, Lemma~\ref{Lemma: Global Sections of L and L|p} implies there exists a section $\tilde{s}_j$ in $H^0(C_j,\mathcal{L}|_{C_j})$ such that $\tilde{s}_j(n_j)\neq 0$ and, if $j>1$, $\tilde{s}_j(n_{j-1})=0$. Thus
\begin{equation*}
    F\left( 0,\ldots, \frac{1}{\tilde{s}_j(n_j)} \tilde{s}_j,\ldots, 0 \right) = (0,\ldots,1,\ldots,0) ,
\end{equation*}
where the vectors above have a zero in every entry except the $i$th one. This shows $F$ is surjective. As a result, $H^1(C,\mathcal{L})$ vanishes.

Tensoring the divisor exact sequence of $C$ for the marked point $x_2$ by $\mathcal{L}$ and pushing forward gives the long exact sequence:
\begin{equation*}
\begin{tikzcd}
    0 \arrow[r]
    & H^0(C,\mathcal{L}(-x_2)) \arrow[r] 
    & H^0(C,\mathcal{L}) \arrow[r]
    & H^0(x_2,\mathcal{L}|_{x_2}) \arrow[r]
    & ~ \\
    ~ \arrow[r]
    & H^1(C,\mathcal{L}(-x_2)) \arrow[r] 
    & 0 .
\end{tikzcd}
\end{equation*}

If $G_{x_2}$ acts nontrivially on the fiber of $\mathcal{L}$ at $x_2$, then $H^0(x_2,\mathcal{L}|_{x_2})$ vanishes. So, the cohomology group $H^1(C,\mathcal{L}(-x_2))$ also vanishes. 

If $G_{x_2}$ acts trivially on the fiber of $\mathcal{L}$ at $x_2$, we need to show that the map $H^0(C,\mathcal{L}) \to H^0(x_2,\mathcal{L}|_{x_2})$ is surjective. By Lemma~\ref{Lemma: Global Sections of L and L|p}, there exists a section $s$ of $\mathcal{L}|_{C_k}\to C_k$ such that $s(n_{k-1})$ is zero and $s(x_2)$ is nonzero. The section $s$ can be extended by the zero section on the other components to get a section in $H^0(C,\mathcal{L})$ that does not vanish at $x_2$. Therefore the map $H^0(C,\mathcal{L}) \to H^0(x_2,\mathcal{L}|_{x_2})$ is surjective, so $H^1(C,\mathcal{L}(-x_2))$ is zero.
\end{proof}

\begin{theorem}
\label{Theorem: Weak convexity}
If a vector bundle $\mathcal{E}\to\mathfrak{X}$ is weakly semi-positive, then it is weakly convex.
\end{theorem}

\begin{proof}
Consider a $0+$-stable quasimap $(C,x_1,x_2,[u])$ in $Q^{0+}_{0,2}(X,\beta)$. Since $[u]^*\mathcal{E}$ is a direct sum of line bundles, we may write
\begin{equation*}
    H^1(C,[u]^*\mathcal{E}(-x_2)) = \bigoplus_{i=1}^r H^1(C,\mathcal{L}_{i}(-x_2)). 
\end{equation*}
Each $\mathcal{L}_{i}$ satisfies the hypothesis of Proposition~\ref{Proposition: Convexity for orbifold hypersurfaces}. Hence, the cohomology group $H^1(C,[u]^*\mathcal{E}(-x_2))$ vanishes.
\end{proof}


\subsection{Weak concavity}
\label{Section: The log canonical bundle}

In this section we define weak concavity. We then prove the relative log canonical bundle over the universal curve $\pi:\mathcal{C}\to \nolinebreak Q_{0,2}^{0+}(X,\beta)$ is trivial and use this result to prove a vector bundle $\mathcal{E}\to\mathfrak{X}$ is weakly convex if and only if its dual $\mathcal{E}^\vee\to\mathfrak{X}$ is weakly concave.


\begin{definition}
A vector bundle $\mathcal{E}^\vee$ over $\mathfrak{X}$ is \emph{weakly concave} if $H^0(C,[u]^*\mathcal{E}^\vee(-x_1))$ vanishes for all $0+$-stable quasimaps $(C,x_1,x_2,[u])\in Q^{0+}_{0,2}(X,\beta)$ and $\beta\in \nolinebreak \Eff(W,G,\theta)$.
\end{definition}

\begin{lemma}
\label{Lemma: The log canonical bundle over a given source curve is trivial}
The log canonical bundle over a source curve $(C,x_1,x_2)$ of a $0+$-stable quasimap $(C,x_1,x_2,[u])$ in $Q^{0+}_{0,2}(X,\beta)$ is trivial.
\end{lemma}

\begin{proof}
If $C$ is smooth, the canonical bundle of $\CP(a,b)/\mu_l$ is isomorphic to $\mathcal{O}(-[0]-[\infty])$. The result is then immediate. When the source curve is nodal, we again use Lemma~\ref{Lemma: Orbifold Source Curves}. In this case we will show that $\omega_{C}(x_1+x_2)$ has a nowhere vanishing section, which gives an isomorphism $\omega_{C}(x_1+ \nolinebreak x_2) \cong \nolinebreak \mathcal{O}_{C}$.

Let $C$ have $k$ rational components and $k-1$ nodes, labeled as $\{C_i\}_{i=1}^k$ and $\{n_i\}_{i=1}^{k-1}$ as in the proof of Proposition~\ref{Proposition: Convexity for orbifold hypersurfaces}. For convenience, relabel $x_1$ by $n_0$ and $x_2$ by $n_k$. Tensoring the normalization sequence \eqref{Equation: Normalization exact sequence} by the log canonical bundle gives rise to the following long exact sequence:
\begin{equation*}
0 \rightarrow H^0(C,\omega_C(x_1+x_2)) \xrightarrow{G} \bigoplus_{i=1}^k H^0(C_i,\omega_{C_i}(n_{i-1}+n_i))  \xrightarrow{F} \bigoplus_{i=1}^{k-1} H^0(n_i,\omega_{C}(x_1+x_2)|_{n_i}) \rightarrow \cdots .
\end{equation*}
    
Note that the map $F$ is a sum of residue maps. There exists nowhere vanishing sections $s_i$ of each $\omega_{C_i}(n_{i-1}+n_i)$ that have residue $1$ near $n_{i-1}$ and $-1$ near $n_i$. The tuple
\begin{equation*}
    \left( s_1 , \ldots , s_k \right) \in \bigoplus_{i=1}^k H^0(C_i,\omega_{C_i}(n_{i-1}+n_i))
\end{equation*}
lies in kernel of $F$, which is isomorphic to $H^0(C,\omega_C(x_1+x_2))$. Hence, the sections $s_1,\ldots,s_k$ glue together to give a nowhere vanishing global section of $\omega_C(x_1+x_2)$. Thus the log canonical bundle is trivial. 
\end{proof}

\begin{proposition}
\label{Proposition: The log canonical bundle over the universal curve is trivial}
The relative log canonical bundle over the universal curve $\pi : \mathcal{C}\to Q_{0,2}^{0+}(X,\beta)$ is trivial.
\end{proposition}
\begin{proof}
Tensoring the divisor exact sequence of $\mathcal{C}$ for the divisor corresponding to the first marked point by $\omega_\pi(x_2)$ and pushing forward, we obtain the long exact sequence
\begin{equation*}
\begin{tikzcd}
    0 \arrow[r]
    & \mathbbm{R}^0\pi_*\omega_\pi(x_2) \arrow[r] 
    & \mathbbm{R}^0\pi_*\omega_{\pi}(x_1+x_2) \arrow[r]
    & \mathbbm{R}^0\pi_*\left(\omega_\pi(x_1+x_2)|_{x_1}\right) \arrow[r]
    & ~ \\
    ~\arrow[r]
    & \mathbbm{R}^1\pi_*\omega_\pi(x_2) \arrow[r] 
    & \mathbbm{R}^1\pi_*\omega_{\pi}(x_1+x_2) \arrow[r]
    & 0 .
\end{tikzcd}
\end{equation*}

Lemma~\ref{Lemma: The log canonical bundle over a given source curve is trivial} and Serre duality imply that for every $0+$-stable quasimap $(C,x_1,x_2,[u])$ in $Q_{0,2}^{0+}(X,\beta)$, the cohomology groups $H^0(C,\omega_C(x_2))$ and $H^1(C,\omega_C(x_2))$ vanish. Therefore $\mathbb{R}^0\pi_*\omega_\pi(x_2)$ and $\mathbb{R}^1\pi_*\omega_\pi(x_2)$ are both zero. Hence, the map
\begin{equation*}
    \mathbbm{R}^0\pi_*\omega_{\pi}(x_1+x_2) \longrightarrow \mathbbm{R}^0\pi_*\left(\omega_\pi(x_1+x_2)|_{x_1}\right) 
\end{equation*}
is an isomorphism. 

The sheaf $\mathbbm{R}^0\pi_*(\omega_\pi(x_1+x_2)|_{x_1})$ is canonically trivialized by the residue map. Thus $\mathbbm{R}^0\pi_*\omega_\pi(x_1+x_2)$ is a trivial line bundle.

The constant function $1\in \Gamma( Q_{0,2}^{0+}(X,\beta), \mathbbm{R}^0\pi_*(\omega_\pi(x_1+x_2)|_{x_1}) )$ is the image of a section $s\in \Gamma( Q_{0,2}^{0+}(X,\beta), \mathbbm{R}^0\pi_*\omega_{\pi}(x_1+x_2) )$. The section $s$ corresponds to a nonzero section $\tilde{s}$ of $\omega_\pi(x_1+x_2)$. The restriction of $\omega_\pi(x_1+\nolinebreak x_2)$ to a fiber $(C,x_1,x_2)$ of the universal curve $\mathcal{C}$ is canonically trivial by Lemma~\ref{Lemma: The log canonical bundle over a given source curve is trivial}. The section $\tilde{s}|_C$ is nonzero and hence nowhere vanishing on each fiber $(C,x_1,x_2)$ in $\mathcal{C}$. Therefore $\tilde{s}$ is  nowhere vanishing.

Hence, $\omega_{\pi}(x_1+x_2)$ has a nowhere vanishing section and so it is trivial. 
\end{proof}

\begin{theorem}
\label{Theorem: Weak concavity}
The vector bundle $\mathcal{E}$ is weakly convex if and only if $\mathcal{E}^\vee$ is weakly concave.
\end{theorem}

\begin{proof}
For a point $(C,x_1,x_2,[u])$ in the moduli space $Q_{0,2}^{0+}(X,\beta)$, the log canonical bundle $\omega_C(x_1+x_2)$ is canonically trivial by Lemma~\ref{Lemma: The log canonical bundle over a given source curve is trivial}. Hence, the canonical bundle $\omega_C$ is isomorphic to $\mathcal{O}_C(-x_1-x_2)$. We now have the following:
\begin{align*}
    H^0(C, [u]^*\mathcal{E}^\vee(-x_1)) &= H^1(C, [u]^*\mathcal{E}(x_1) \otimes \omega_C)^\vee \\
    &= H^1(C, [u]^*\mathcal{E}(-x_2))^\vee .
\end{align*}
The first equality is Serre duality and the second is because of the isomorphism $\omega_C\cong \mathcal{O}_C(-x_1-x_2)$. If the left hand side vanishes for all $0+$-stable two-pointed quasimaps then so does the right, and vice versa.
This completes the proof.
\end{proof}


\section{Quantum Lefschetz}
\label{Section: Quantum Lefschetz}

The quantum Lefschetz hyperplane theorem compares the genus-zero Gromov-Witten theory of a space $X$ with that of a complete intersection $Z\subset X$ defined by a section of a vector bundle $E \to X$ \cite{Coates_Givental_Quantum_Riemann-Roch_Lefschetz_and_Serre_2007,batyrev2000mirror,kim1999quantum,bertram2000another,lee2001quantum,gathmann2003relative}. The same proof applies to quasimaps under certain conditions on $E$ \cite{Ciocan_Fontanine_Kim_Maulik_Stable_quasimaps_to_GIT_quotients_2014}.

In this section we use a modification of the quantum Lefschetz theorem for two-pointed genus-zero quasimaps to relate the generating function $L^{Z,0+}(z)$ to a twisted version of $L^{X,0+}(z)$.


\begin{remark}[quantum Lefschetz for orbifolds]\label{r:conv}
Let $Z$ be a closed subset of $X$ cut out by a section of a vector bundle $E\to X$. In the case when $X$ is a smooth variety, the quantum Lefschetz theorem holds as long as $E$ is the direct sum of semi-positive line bundles. However this can fail when $X$ is an orbifold. For example, the line bundle  $\mathcal{O}_{\CP(1,1,2,2)}(1) \to \CP(1,1,2,2)$ is positive, but quantum Lefschetz fails for the hypersurface $\CP(1,2,2)$ defined by a section of $\mathcal{O}_{\CP(1,1,2,2)}(1)$  \cite{coates2012quantum}. 

As observed in \cite{coates2012quantum}, the general setting in which one should expect a quantum Lefschetz statement is not when $E$ is semi-positive, but rather when $E$ is \emph{convex}, that is, when for every stable map $f: C \to X$ from a genus-zero curve $C$, the cohomology group $H^1(C, f^*(E))$ is zero. If $X$ is a variety, then if $E \to X$ is semi-positive it is also convex. This no longer holds when $X$ is an orbifold.  

For two-pointed genus-zero quasimaps, we will see that in fact weak convexity is sufficient for a quantum Lefschetz statement.
\end{remark}

\subsection{Weak convexity and twisted invariants}

Let $X$, $\mathfrak{X}$, and $\mathcal{E}$ be as in \S\ref{Section: A local complete intersections in X}. For the remainder of the paper we assume that $X$ is proper and that $\mathcal{E}$ is weakly convex.

\begin{lemma}
\label{Lemma: H^1(E) vanishes}
The cohomology group $H^1(C,[u]^*\mathcal{E})$ vanishes for all points $(C,x_1,x_2,[u])$ in $Q^{0+}_{0,2}(X,\beta)$ and classes $\beta\in \Eff(W,G,\theta)$.
\end{lemma}

\begin{proof}
Fix a point $(C,x_1,x_2,[u])$ in $Q^{0+}_{0,2}(X,\beta)$. Consider the divisor exact sequence
\begin{equation*}
\begin{tikzcd}
    0 \arrow[r]
    & \mathcal{O}_C(-x_2) \arrow[r]
    & \mathcal{O}_C \arrow[r]
    & \mathcal{O}_C|_{x_2} \arrow[r]
    & 0 .
\end{tikzcd}
\end{equation*}
Tensoring this by $[u]^*\mathcal{E}$ and pushing forward gives rise to a long exact sequence in cohomology
\begin{equation*}
\begin{tikzcd}
    0 \arrow[r]
    & H^0(C,[u]^*\mathcal{E}(-x_2)) \arrow[r] 
    & H^0(C,[u]^*\mathcal{E}) \arrow[r]
    & H^0(x_2,[u]^*\mathcal{E}|_{x_2}) \arrow[r]
    & ~ \\
    ~ \arrow[r]
    & H^1(C,[u]^*\mathcal{E}(-x_2)) \arrow[r] 
    & H^1(C,[u]^*\mathcal{E}) \arrow[r]
    & 0 . ~~~~~~~~~~
\end{tikzcd}
\end{equation*}
By Theorem~\ref{Theorem: Weak convexity}, $H^1(C,[u]^*\mathcal{E}(-x_2))$ vanishes. Therefore $H^1(C,[u]^*\mathcal{E})$ also vanishes.
\end{proof}

Denote by $\pi$ the projection from the universal curve $\mathcal{C}$ to $Q^{0+}_{0,2}(X,\beta)$. Let $[u]$ be the universal map from $\mathcal{C}$ to $\mathfrak{X}$. By Lemma~\ref{Lemma: H^1(E) vanishes}, $\mathbbm{R}^0\pi_*[u]^*\mathcal{E}$ is a vector bundle on $Q_{0,2}^{0+}(X,\beta)$. Let $0_{X}$ denote the zero section of $\mathbbm{R}^0\pi_*[u]^*\mathcal{E}$ and $\tilde{s}\in \Gamma(Q_{0,2}^{0+}(X,\beta), \mathbbm{R}^0\pi_*[u]^*\mathcal{E})$ denote the section induced by $s$.

\begin{theorem}
\label{Theorem: Functoriality of the virtual class}
There is a fiber square
\begin{center}
\begin{tikzcd}
    Q^{0+}_{0,2}(Z,\beta) \arrow[r,hook,"\tilde{j}"] \arrow[d,"\tilde{j}"] 
    & Q^{0+}_{0,2}(X,\beta) \arrow[d,"\tilde{s}"] \\
    Q^{0+}_{0,2}(X,\beta) \arrow[r,hook,"~~~0_{X}" below]
    & \mathbbm{R}^0\pi_*[u]^*\mathcal{E} .
\end{tikzcd}
\end{center}
The virtual classes are related by
\begin{equation*}
    0^!_{X}\left[Q^{0+}_{0,2}(X,\beta)\right]^{\vir} = \left[Q^{0+}_{0,2}(Z,\beta)\right]^{\vir}.
\end{equation*}
\end{theorem}

\begin{proof}
The proof of the theorem in \cite{kim2003functoriality} for Gromov-Witten theory extends to quasimaps.
\end{proof}

Theorem~\ref{Theorem: Functoriality of the virtual class} implies the more familiar quantum Lefschetz statement from \cite{Ciocan_Fontanine_Kim_Maulik_Stable_quasimaps_to_GIT_quotients_2014}.

\begin{proposition}
\emph{\cite{Ciocan_Fontanine_Kim_Maulik_Stable_quasimaps_to_GIT_quotients_2014}}
\label{Proposition: Quantum Lefschetz}
The following classes are equal in the Chow group of $Q_{0,2}^{0+}(X,\beta)$:
\begin{equation*}
    \tilde{j}_*\left[Q^{0+}_{0,2}(Z,\beta)\right]^{\vir} = e(\mathbbm{R}^{0}\pi_*[u]^*\mathcal{E}) \cap \left[Q^{0+}_{0,2}(X,\beta)\right]^{\vir} .
\end{equation*}
\end{proposition}


The definitions and arguments that follow are similar to those appearing in \parencite[\S 2.1]{Pandharipande_after_Givental_1998} to define a twisted quantum product.

\begin{definition}
\label{Definition: Z-twisted virtual class}
In the Chow group of $Q^{0+}_{0,2}(X,\beta)$, define the twisted virtual class
\begin{equation}
\label{Equation: Defintion of the twisted virtual class for the hypersurface}
    \left[Q^{0+}_{0,2}(X,\beta)\right]^{\vir}_{X/Z,2} := e(\mathbbm{R}^0\pi_*[u]^*\mathcal{E}(-x_2)) \cap \left[Q^{0+}_{0,2}(X,\beta)\right]^{\vir} .
\end{equation}
Define the operator $L^{X/Z,0+}(z)$ by
\begin{equation*}
    L^{X/Z,0+}(z)(\alpha) := \alpha + \sum_{\beta\in\Eff} q^\beta \tilde{ev}_{2*}\left( \frac{ev_1^*(\alpha)}{-z-\psi_1} \cap \left[Q^{0+}_{0,2}(X,\beta)\right]^{\vir}_{X/Z,2} \right)
\end{equation*}
for $\alpha\in H^*_{\CR}(X)$.
\end{definition}

Now, we compare $L^{X/Z,0+}(z)$ and $L^{Z,0+}(z)$.

\begin{proposition}\label{Proposition: Pullback of L^X/Z is equivalent to L^Z of the pullback}
The operators $L^{X/Z,0+}(z)$ and $L^{Z,0+}(z)$ are related as follows 
\begin{equation*}
    j^* \circ L^{X/Z,0+}(z) = L^{Z,0+}(z) \circ j^* .
\end{equation*}
\end{proposition}

The analogous statement in Gromov-Witten theory appears in \parencite[Proposition~2.4]{Iritani-Quantum_Cohomology_and_periods_2011}. The proof in this case is very similar.

\begin{proof}
To avoid confusion, we use superscripts to distinguish the evaluation maps $ev_i^X: Q_{0,2}^{0+}(X,\beta)\to \bar IX$ and $ev_i^Z: Q_{0,2}^{0+}(Z,\beta)\to \bar IZ$. For $\alpha$ in $H^*(\bar IX)$, we can express $j^* \circ L^{X/Z,0+}(z)(\alpha)$ as
\begin{equation}
\label{Equation: Proof that j^* is a homomorphism - LHS}
   j^*\alpha + \sum_{\beta\in\Eff} q^\beta j^* \tilde{ev}^{X}_{2*}\left( \frac{ev_1^{X*}(\alpha)}{-z-\psi_1} \cup e(\mathbbm{R}^0\pi_*[u]^*\mathcal{E}(-x_2)) \cap \left[Q^{0+}_{0,2}(X,\beta)\right]^{\vir} \right)
\end{equation}
and $L^{Z,0+}(z) \circ j^* (\alpha)$ as
\begin{equation}
    \label{Equation: Proof that j^* is a homomorphism - RHS}
   j^*(\alpha) + \sum_{\beta\in\Eff} q^\beta \tilde{ev}^{Z}_{2*} \left( \frac{ev_1^{Z*}j^*(\alpha)}{-z-\psi_1} \cap \left[Q^{0+}_{0,2}(Z,\beta)\right]^{\vir} \right) .
\end{equation}
To prove the claim we compare \eqref{Equation: Proof that j^* is a homomorphism - LHS} and \eqref{Equation: Proof that j^* is a homomorphism - RHS} term by term. 

Define $Q$ to be the zero locus of the section $ev_2^{X*}s\in \Gamma(Q_{0,2}^{0+}(X,\beta),ev_2^{X*}E)$ and denote the evaluation maps of $Q$ by $ev_1^Q: Q\to \bar I X$ and $ev_2^Q: Q\to \bar I Z$. Label the following projections $\pi_1^X,\pi_2^X: \bar IX \times \bar IX \to \bar IX$, $\pi_1:\bar IX \times \nolinebreak \bar IZ \to \nolinebreak \bar IX$, and $\pi_2:\bar IX\times\bar IZ\to\bar IZ$. Consider the following diagram, where the rectangles are fiber squares
\begin{equation}
\begin{tikzcd}
    \label{Diagram: Proof that j^* is a homomorphism - fiber diagram for IXxIX and IXxIZ}
    Q_{0,2}^{0+}(Z,\beta) \arrow[r,"j''"] \arrow[rd,"\underline{ev}^Z"]
    &Q \arrow[r,"j'"] \arrow[d,"\underline{ev}^Q"]
    &  Q^{0+}_{0,2}(X,\beta) \arrow[d,"\underline{ev}^X"] \\
    &\bar IX \times \bar IZ \arrow[r,"\id\times j"] \arrow[d, "\pi_2"]
    & \bar IX\times \bar IX \arrow[d,"\pi_2^X"]\\
    &\bar IZ \arrow[r,"j"]
    & \bar IX ,
\end{tikzcd}
\end{equation}
where $\underline{ev}^X=(ev_1^X,ev_2^X)$, $\underline{ev}^Q=(ev_1^Q,ev_2^Q)$, and $\underline{ev}^Z=(j\circ ev_1^Z,ev_2^Z)$.

Using the diagram above, we rewrite each term in the sum of \eqref{Equation: Proof that j^* is a homomorphism - LHS}:
\begin{align}
    &j^* \tilde{ev}^X_{2*} \left( \frac{ev_1^{X*}(\alpha)}{-z-\psi_1} \cup e(\mathbbm{R}^0\pi_*[u]^*\mathcal{E}(-x_2)) \cap \left[Q^{0+}_{0,2}(X,\beta)\right]^{\vir} \right) \nonumber \\
    =& \tilde{ev}^Q_{2*} j^! \left( \frac{ev_1^{X*}(\alpha)}{-z-\psi_1} \cup e(\mathbbm{R}^0\pi_*[u]^*\mathcal{E}(-x_2)) \cap \left[Q^{0+}_{0,2}(X,\beta)\right]^{\vir} \right) \nonumber \\
    =& \iota_* \pi_{2*} \underline{ev}^Q_{*} (\id\times j)^! \left( \frac{\underline{ev}^{X*}\pi^{X*}_1(\alpha)}{-z-\psi_1} \cup e(\mathbbm{R}^0\pi_*[u]^*\mathcal{E}(-x_2)) \cap \left[Q^{0+}_{0,2}(X,\beta)\right]^{\vir} \right) \nonumber \\
    =& \iota_* \pi_{2*} (\id\times j)^* \underline{ev}^X_{*} \left( \frac{\underline{ev}^{X*}\pi^{X*}_1(\alpha)}{-z-\psi_1} \cup e(\mathbbm{R}^0\pi_*[u]^*\mathcal{E}(-x_2)) \cap \left[Q^{0+}_{0,2}(X,\beta)\right]^{\vir} \right) \nonumber \\
    =& \iota_* \pi_{2*} \left( (\id\times j)^* \pi^{X*}_1(\alpha) \cup (\id\times j)^*\underline{ev}^X_*\left( \frac{e(\mathbbm{R}^0\pi_*[u]^*\mathcal{E}(-x_2)) \cap \left[Q^{0+}_{0,2}(X,\beta)\right]^{\vir}}{-z-\psi_1} \right)  \right) \nonumber \\
    =& \iota_* \pi_{2*} \left( \pi^{*}_1(\alpha) \cup \underline{ev}^Q_* j^! \left( \frac{e(\mathbbm{R}^0\pi_*[u]^*\mathcal{E}(-x_2)) \cap \left[Q^{0+}_{0,2}(X,\beta)\right]^{\vir}}{-z-\psi_1} \right)  \right) .
    \label{Equation: Proof that j^* is a homomorphism - LHS simplified}
\end{align}
Equalities one, three, and five follow from \parencite[Theorem~6.2(a)]{Fulton_Intersection_theory} and \eqref{Diagram: Proof that j^* is a homomorphism - fiber diagram for IXxIX and IXxIZ}. Equality two is by \parencite[Theorem~6.2(c)]{Fulton_Intersection_theory} and \eqref{Diagram: Proof that j^* is a homomorphism - fiber diagram for IXxIX and IXxIZ}. Equality four is the projection formula.

Let $0_X$ denote the zero section of $\mathbbm{R}^0\pi_*[u]^*\mathcal{E}$ over $Q^{0+}_{0,2}(X,\beta)$. We rewrite each term in~\eqref{Equation: Proof that j^* is a homomorphism - RHS}:
\begin{align}
    &\tilde{ev}^{Z}_{2*} \left( \frac{ev_1^{Z*}j^*(\alpha)}{-z-\psi_1} \cap \left[Q^{0+}_{0,2}(Z,\beta)\right]^{\vir} \right) \nonumber \\
    =& \iota_* \pi_{2*} \underline{ev}^{Z}_{*} \left( \frac{\underline{ev}^{Z*}\pi_1^*(\alpha)}{-z-\psi_1} \cap \left[Q^{0+}_{0,2}(Z,\beta)\right]^{\vir} \right) \nonumber \\
    =& \iota_* \pi_{2*} \left( \pi_1^*(\alpha) \cup \underline{ev}^{Z}_{*} \left( \frac{\left[Q^{0+}_{0,2}(Z,\beta)\right]^{\vir}}{-z-\psi_1} \right) \right) \nonumber \\
    =& \iota_* \pi_{2*} \left( \pi_1^*(\alpha) \cup \underline{ev}^{Q}_{*} j''_* 0^!_X\left( \frac{\left[Q^{0+}_{0,2}(X,\beta)\right]^{\vir}}{-z-\psi_1} \right) \right) .
    \label{Equation: Proof that j^* is a homomorphism - RHS simplified}
\end{align}
The third equality follows by factoring $\underline{ev}^Z$ as $\underline{ev}^Q\circ\nolinebreak j''$, Theorem~\ref{Theorem: Functoriality of the virtual class}, and \parencite[Proposition~6.3]{Fulton_Intersection_theory}.

Compare \eqref{Equation: Proof that j^* is a homomorphism - LHS simplified} and \eqref{Equation: Proof that j^* is a homomorphism - RHS simplified}. We may factor $\psi$-classes out of Gysin maps by \parencite[Proposition~6.3]{Fulton_Intersection_theory}. Hence, to complete the proof it suffices to show the following equality in the Chow group of $Q$:
\begin{equation}
    \label{Equation: Proof that j^* is a homomorphism - final equality}
    j''_{*} 0^{!}_{X} \left( \left[Q^{0+}_{0,2}(X,\beta)\right]^{\vir} \right)
    = j^! \left( e(\mathbbm{R}^0\pi_*[u]^*\mathcal{E}(-x_2)) \cap \left[Q^{0+}_{0,2}(X,\beta)\right]^{\vir} \right) .
\end{equation}

Consider the fiber diagram (also constructed in \parencite[\S 2]{Iritani-Quantum_Cohomology_and_periods_2011}):
\begin{equation}
\label{Diagram: Proof that j^* is a homomorphism - Iritani's fiber diagram}
\begin{tikzcd}
    Q^{0+}_{0,2}(Z,\beta) \arrow[r,"j''"] \arrow[d,"j''"]
        & Q \arrow[r] \arrow[d,"\tilde{s}_Q"]
        & Q^{0+}_{0,2}(X,\beta) \arrow[dd,"\tilde{s}"] \\
    Q \arrow[r, "0_{Q}"] \arrow[d]
        & \mathbbm{R}^0\pi_*[u]^*\mathcal{E}(-x_2)|_Q \arrow[d,"f"] \\
    Q^{0+}_{0,2}(X,\beta) \arrow[r,"0'_{X}"]
        & \mathbbm{R}^0\pi_*[u]^*\mathcal{E}(-x_2) \arrow[r,"h"]
        & \mathbbm{R}^0\pi_*[u]^*\mathcal{E} .
\end{tikzcd}
\end{equation}
The morphisms $0_Q$ and $0'_X$ are the zero sections, $f$ and $h$ are the natural inclusions, and $\tilde{s}_Q$ is the section induced by $s$. We then have the following:
\begin{align*}
    j''_* 0^{!}_{X} \left( \left[Q^{0+}_{0,2}(X,\beta)\right]^{\vir} \right)
    &= j''_* 0'^{!}_{X} h^! \left( \left[Q^{0+}_{0,2}(X,\beta)\right]^{\vir} \right) \\
    &= 0^*_{Q} \tilde{s}_{Q*} h^! \left( \left[Q^{0+}_{0,2}(X,\beta)\right]^{\vir} \right) \\
    &= h^! \left( e(\mathbbm{R}^0\pi_*[u]^*\mathcal{E}(-x_2)) \cap \left[Q^{0+}_{0,2}(X,\beta)\right]^{\vir} \right) .
\end{align*}
The first equality is \parencite[Theorem~6.5]{Fulton_Intersection_theory}. The second is \parencite[Theorem~6.2(a)]{Fulton_Intersection_theory}.

To verify \eqref{Equation: Proof that j^* is a homomorphism - final equality} it now suffices to show that $h^!$ equals $j^!$. This can be seen from the following commutative diagram, where the front, back, top, and bottom faces are fiber squares: 
\begin{equation*}
\begin{tikzcd}
    & Q  \arrow[rr,"j'"] \arrow[dd] \arrow[dl,"f \circ \tilde{s}_Q",swap] 
    & & Q^{0+}_{0,2}(X,\beta)  \arrow[dd,"ev_2^{X}"] \arrow[dl,"\tilde{s}"] \\
    \mathbbm{R}^0\pi_*[u]^*\mathcal{E}(-x_2)  \arrow[rr,crossing over,"h",pos=.63] \arrow[dd] 
    & & \mathbbm{R}^0\pi_*[u]^*\mathcal{E} \arrow[dd,"\tilde{s}|_{x_2}",pos=.25] \\
    & \bar IZ  \arrow[rr,"j",pos=.3] \arrow[dl,"j",swap] 
    & &  \bar IX  \arrow[dl,"s",swap] \\
    \bar IX \arrow[rr,"0_{E}"] 
    & & \bar IE . \ar[from=uu,crossing over]
\end{tikzcd}
\end{equation*}
The morphism $\tilde{s}|_{x_2}$ is defined by the composition 
\begin{equation*}
    \mathbbm{R}^0\pi_*[u]^*\mathcal{E} \longrightarrow ev_2^{X*}E \longrightarrow \bar IE .
\end{equation*}

Applying \parencite[Theorem~6.2(c)]{Fulton_Intersection_theory} twice completes the proof,
\begin{align*}
    &h^! \left( e(\mathbbm{R}^0\pi_*[u]^*\mathcal{E}(-x_2)) \cap \left[Q^{0+}_{0,2}(X,\beta)\right]^{\vir} \right) \\
    =& 0_E^! \left( e(\mathbbm{R}^0\pi_*[u]^*\mathcal{E}(-x_2)) \cap \left[Q^{0+}_{0,2}(X,\beta)\right]^{\vir} \right) \\
    =& j^! \left( e(\mathbbm{R}^0\pi_*[u]^*\mathcal{E}(-x_2)) \cap \left[Q^{0+}_{0,2}(X,\beta)\right]^{\vir} \right) .
\end{align*}
\end{proof}


\section{The total space}
\label{Section: The total space}
In this section we consider quasimaps to $E^\vee$.  Similar to the previous section, we will compute the quasimap invariants of $E^\vee$ in terms of integrals over the moduli space of $0+$-stable quasimaps to $X$.

\begin{definition}\label{Definition: qmaps to E dual}
Let $\mathcal{E}^\vee \to \mathfrak{X}$ be a vector bundle. Define the moduli space of $k$-pointed, genus-$g$, $0+$-stable quasimaps to $E^\vee$ of degree $\beta$ as
$$Q_{g,k}^{0+}(E^\vee,\beta):=\operatorname{tot}(\pi_* [u]^* \mathcal{E}^\vee ):= \Spec  \operatorname{Sym} \mathbbm{R}^1 \pi_* \left([u]^* \mathcal{E} \otimes \omega_{\pi}\right) .$$
There is a natural perfect obstruction theory on $Q_{g,k}^{0+}(E^\vee,\beta)$ obtained by pulling back the obstruction theory on $Q_{g,k}^{0+}(X,\beta)$ and taking the direct sum with $(\mathbbm{R}^\bullet \pi_*[u]^* \mathcal{E}^\vee)^\vee$.  This yields a virtual fundamental class $[Q_{g,k}^{0+}(E^\vee,\beta)]^{\vir}$.  When the evaluation maps 
$$ev_i: Q_{g,k}^{0+}(E^\vee,\beta) \to \bar I E^\vee$$
are proper, we can define quasimap invariants as in Definition~\ref{nonproper brackets}.
\end{definition}

As in the previous section we will assume that $\mathcal E$ is weakly convex. Then recall by Theorem~\ref{Theorem: Weak concavity} that $\mathcal{E}^\vee \to \mathfrak{X}$ is weakly concave.

On the universal curve over $Q_{0,2}^{0+}(X,\beta)$, consider the divisor given by the first marked point. Tensor the divisor exact sequence by $[u]^*\mathcal{E}^\vee$ and apply $\mathbbm{R}\pi_*(-)$ to get the long exact sequence
\begin{equation}
\label{Equation: SES for Euler class relation for L^X/Y(z) proof}
\begin{tikzcd}
    &0 \arrow[r]
    & \mathbbm{R}^0\pi_*[u]^*\mathcal{E}^\vee \arrow[r]
    & \mathbbm{R}^0\pi_*[u]^*\mathcal{E}^\vee|_{x_1} \arrow[r]
    & ~ \\
    ~ \arrow[r]
    & \mathbbm{R}^1\pi_*[u]^*\mathcal{E}^\vee(-x_1) \arrow[r]
    & \mathbbm{R}^1\pi_*[u]^*\mathcal{E}^\vee \arrow[r]
    & 0 .
\end{tikzcd}    
\end{equation}

Fix an open and closed subspace $ev_1^{-1}(X_g^\vee) \subset Q^{0+}_{0,2}(X,\beta)$.  Let $G_{x_1}$ denote the isotropy group at $x_1$ for a point $(C,x_1,x_2,[u])$ in this subspace.  Recall that $[u]^*\mathcal{E}$ splits as a direct sum of line bundles $\oplus_{i=1}^r \mathcal{L}_i$.  The line bundle $\mathbbm{R}^0\pi_*\mathcal{L}_{i}^\vee|_{x_1}$ is nonzero if and only if $G_{x_1}$ acts trivially on the fiber of $\mathcal{L}_i^\vee$ at $x_1$.  We conclude that on a given open and closed subset $ev_1^{-1}(X_g^\vee) \subset Q^{0+}_{0,2}(X,\beta)$, 
\begin{equation}
\label{e: restricted to x1} 
    \mathbbm{R}^0\pi_*[u]^*\mathcal{E}^\vee|_{x_1} = ev_1^*E_g^\vee.
\end{equation}

Combining \eqref{Equation: SES for Euler class relation for L^X/Y(z) proof} and \eqref{e: restricted to x1}, we obtain the following equality in $K$-theory:
\begin{equation}
\label{e: K theory equality}
    ev_1^*E_g^\vee \ominus \mathbbm{R}\pi_*[u]^*\mathcal{E}^\vee = \mathbbm{R}^1\pi_*[u]^*\mathcal{E}^\vee(-x_1) .
\end{equation}

\begin{proposition}
If $\mathcal E^\vee$ is weakly concave, then the evaluation maps $$ev_i: Q^{0+}_{0,2}(E^\vee,\beta) \to \bar I E^\vee$$ are proper.
\end{proposition}
\begin{proof}
For notational simplicity we consider $ev_1$.  By \eqref{Equation: SES for Euler class relation for L^X/Y(z) proof} and \eqref{e: restricted to x1}, there is a closed immersion
$$Q_{g,k}^{0+}(E^\vee,\beta)=\operatorname{tot}(\pi_* [u]^* \mathcal{E}^\vee )\hookrightarrow  \operatorname{tot}(\pi_*([u]^*\mathcal{E}^\vee|_{x_1})) = (ev_1^X)^* \bar I E^\vee,$$ where $(ev_1^X)^* \bar I E^\vee$  denotes the fiber product of  $ev_1^X: Q_{g,k}^{0+}(X,\beta) \to \bar I X$ and $\bar IE^\vee \to \bar IX$.  The evaluation map $ev_1^{E^\vee}$ factors as
$$Q_{g,k}^{0+}(E^\vee,\beta) \hookrightarrow (ev_1^X)^* \bar I E^\vee \to \bar I E^\vee.$$  This composition is  proper because $ev_1^X$ is.
\end{proof}

By weak concavity, the sheaf $\mathbbm{R}^1\pi_*[u]^*\mathcal{E}^\vee(-x_1)$ is a vector bundle on $Q_{0,2}^{0+}(X,\beta)$. This allows the following definition.
\begin{definition}
\label{Definition: Y-twisted virtual class}
Define the $\mathcal{E}^\vee$-twisted virtual class to be
\begin{equation}
\label{Equation: Defintion of the twisted virtual class for the total space}
    \left[Q^{0+}_{0,2}(X,\beta)\right]^{\vir}_{X/E^\vee,1} := e(\mathbbm{R}^1\pi_*[u]^*\mathcal{E^\vee}(-x_1)) \cap \left[Q^{0+}_{0,2}(X,\beta)\right]^{\vir} .
\end{equation}
Define the operator $L^{X/E^\vee,0+}(z)$ by
\begin{align*}
    L^{X/E^\vee,0+}(z)(\alpha) := \alpha + \sum_{\beta \in\Eff} q^\beta \tilde{ev}_{2*}\left( \frac{ev_1^*(\alpha)}{-z-\psi} \cap \left[Q^{0+}_{0,2}(X,\beta)\right]^{\vir}_{X/E^\vee,1} \right) ,
\end{align*}
for $\alpha\in H^*_{\CR}(X)$.
\end{definition}

We will make use of the commutative diagram
\begin{equation}\label{comD}
\begin{tikzcd}
   Q^{0+}_{0,2}(X,\beta) \arrow[r, "\tilde i"] \arrow[d,"ev_i^X"]
    &  Q^{0+}_{0,2}(E^\vee,\beta) \arrow[d,"ev_i^{E^\vee}"] 
     \\
    \bar IX \arrow[r, hook, "i"]
    & \bar I E^\vee ,
\end{tikzcd}
\end{equation}
where the $i$th evaluation maps on $Q^{0+}_{0,2}(X,\beta)$ and $Q^{0+}_{0,2}(E^\vee,\beta)$ are denoted by $ev^X_i$ and $ev_i^{E^\vee}$ respectively.

\begin{proposition}\label{prop: this fits here}
Given $\gamma_1, \gamma_2 \in H^*_{CR, \ct}(E^\vee)$, choose $\alpha_1, \alpha_2 \in H^*_{CR}(X)$ such that $\gamma_i = i_*(\alpha_i)$.  We have the following equality:
\begin{equation}
\label{Equation: Two-pointed invariant of E-dual first}
    \left\langle \gamma_1 \psi_1^{a_1}, \gamma_2 \psi_2^{a_2} \right\rangle^{E^\vee,0+}_{0,\beta}
    = \int_{\left[Q_{0,2}^{0+}(X,\beta)\right]^{\vir}_{X/E^\vee,1}} ev_1^{X *} (\alpha_1) \psi_1^{a_1} \cup ev_2^{X *} (\alpha_2 \cup e(E^\vee_{g_2}) ) \psi_2^{a_2} .
\end{equation}
\end{proposition}

\begin{proof}
The claim is most easily seen via virtual localization.  
Consider the $\C^*$-action on $\mathcal{E}^\vee$ given by scaling fibers.  This induces actions on $\bar IE^\vee$ and $Q^{0+}_{0,2}(E^\vee,\beta)$, with fixed loci $\bar IX$ and  $Q^{0+}_{0,2}(X,\beta)$ respectively.  All the maps in Diagram~\eqref{comD} are $\C^*$-equivariant.
By~(1) of~\cite{Graber-Pand}, we have the equality
\begin{equation}\label{e:localization} \left[Q^{0+}_{0,2}(E^\vee,\beta)\right]^{\C^*, \vir} = \tilde i_* \left( \frac{\left[Q^{0+}_{0,2}(X,\beta)\right]^{\C^*, \vir}}{e_{\C^*}(\mathbbm{R}\pi_*[u]^*\mathcal{E}^\vee )} \right),\end{equation}
where $[ - ]^{\C^*, \vir}$ denotes the $\C^*$-equivariant virtual fundamental class.

Choose equivariant lifts of $\gamma_1$, $\gamma_2$, $\alpha_1$, and $\alpha_2$ (by abuse of notation we will not change their labels).  Using the fact that $\gamma_2$ is the pushforward of a class $\alpha_2$ supported on $IX$, a localization argument shows that the left hand side of \eqref{Equation: Two-pointed invariant of E-dual first} is equal to the non-equivariant limit of 
$$\int_{\left[Q_{0,2}^{0+}(E^\vee,\beta)\right]^{\C^*,\vir}} ev_1^{E^\vee *} (\gamma_1)\psi_1^{a_1} \cup ev_2^{E^\vee *} (\gamma_2) \psi_2^{a_2}.$$
This can be rewritten as
\begin{align*}
    & \int_{\left[Q_{0,2}^{0+}(X,\beta)\right]^{\C^*,\vir}} \frac{ev_1^{X *} (\alpha_1 \cup e_{\C^*}(E^\vee_g) ) \psi_1^{a_1} \cup ev_2^{X *} (\alpha_2 \cup e_{\C^*}(E^\vee_g) ) \psi_2^{a_2}}{e_{\C^*}(\mathbbm{R}\pi_*[u]^*\mathcal{E}^\vee)} \\
    =& \int_{\left[Q_{0,2}^{0+}(X,\beta)\right]^{\C^*,\vir}} ev_1^{X *} (\alpha_1) \psi_1^{a_1} \cup ev_2^{X *} (\alpha_2 \cup e_{\C^*}(E^\vee_g) ) \psi_2^{a_2} \cup \frac{e_{\C^*}(ev_1^{X*}E^\vee_g)}{e_{\C^*}(\mathbbm{R}\pi_*[u]^*\mathcal{E}^\vee)} \\
    =& \int_{\left[Q_{0,2}^{0+}(X,\beta)\right]^{\C^*,\vir}} ev_1^{X *} (\alpha_1) \psi_1^{a_1} \cup ev_2^{X *} (\alpha_2 \cup e_{\C^*}(E^\vee_g) ) \psi_2^{a_2} \cup e_{\C^*}(\mathbbm{R}^1\pi_*[u]^*\mathcal{E}^\vee(-x_1)) .
\end{align*}
The first line is a result of \eqref{e:localization}, the projection formula, and \eqref{comD}. The second is immediate. The third line follows from \eqref{e: K theory equality}.
Taking the non-equivariant limit yields the desired equality.
\end{proof}
A similar argument yields the following.
\begin{proposition}
\label{Proposition: Pushforward of L^X/Y is equivalent to L^Y of the pushforward}
The operators $L^{X/E^\vee,0+}(z)$ and $L^{E^\vee,0+}(z)$ are related by  pushforward along $i$,
\begin{equation}
\label{e: L commutes with i_*}
    i_* \circ L^{X/E^\vee,0+}(z) = L^{E^\vee,0+}(z) \circ i_* .
\end{equation}
\end{proposition}

\begin{proof}
For $\alpha\in H^*_{\CR}(X)$, we can rewrite the left-hand-side of \eqref{e: L commutes with i_*} applied to $\alpha$ as
\begin{equation}
\label{Equation: Proof that i_* is a homorphism - LHS}
    i_* \circ L^{X/E^\vee,0+}(z) (\alpha) = i_* (\alpha) + \sum_{\beta \in\Eff} q^\beta i_* \tilde{ev}^{E^\vee}_{2*}\left( \frac{ev^{*}_{1} (\alpha)}{-z-\psi} \cap \left[Q^{0+}_{0,2}(X,\beta)\right]^{\vir}_{X/E^\vee,1} \right)
\end{equation}
and the right-hand-side as
\begin{equation}
\label{Equation: Proof that i_* is a homorphism - RHS}
    L^{E^\vee,0+}(z) \circ i_* (\alpha) = i_* (\alpha) + \sum_{\beta \in\Eff} q^\beta \tilde{ev}^{X}_{2*}\left( \frac{ev^{*}_{1} i_* (\alpha)}{-z-\psi} \cap \left[Q^{0+}_{0,2}(E^\vee,\beta)\right]^{\vir} \right).
\end{equation}
  

We equate \eqref{Equation: Proof that i_* is a homorphism - LHS} and \eqref{Equation: Proof that i_* is a homorphism - RHS} term by term.  Again we apply via virtual localization.
Fix a degree $\beta$ and choose a class $\alpha$ in $H^*_{\CR, \C^*}(X)$.  
 
Consider the following sequence of equalities in the localized equivariant cohomology ring of $\bar IE^\vee$:
\begin{align*}
   & \tilde{ev}^{E^\vee}_{2*}\left( \frac{ev^{E^{\vee}*}_{1} i_* (\alpha)}{-z-\psi} \cap \left[Q^{0+}_{0,2}(E^\vee,\beta)\right]^{\C^*, \vir} \right) \\
    =&  \tilde{ev}^{E^\vee}_{2*}\left( \frac{ev^{E^{\vee}*}_{1} i_* (\alpha)}{-z-\psi} \cap\tilde i_* \left( \frac{\left[Q^{0+}_{0,2}(X,\beta)\right]^{\C^*, \vir}}{e_{\C^*}(\mathbbm{R}\pi_*[u]^*\mathcal{E}^\vee )} \right)\right) \\
    =& i_*\tilde{ev}^X_{2*}\left( \frac{ev^{X*}_{1}(\alpha \cup e_{\C^*}(E^\vee_g) )}{-z-\psi} \cap  \frac{\left[Q^{0+}_{0,2}(X,\beta)\right]^{\C^*, \vir}}{e_{\C^*}(\mathbbm{R}\pi_*[u]^*\mathcal{E}^\vee )} \right) \\
    =& i_*\tilde{ev}^X_{2*}\left( \frac{ev^{X*}_{1}(\alpha)}{-z-\psi}  \cup \frac{e_{\C^*}(ev_1^{X*}E^\vee_g)}{e_{\C^*}(\mathbbm{R}\pi_*[u]^*\mathcal{E}^\vee )} \cap \left[Q^{0+}_{0,2}(X,\beta)\right]^{\C^*, \vir} \right) \\
    =& i_*\tilde{ev}^X_{2*}\left( \frac{ev^{X*}_{1}(\alpha)}{-z-\psi} \cup e_{\C^*}(\mathbbm{R}^1\pi_*[u]^*\mathcal{E}^\vee(-x_1 )) \cap \left[Q^{0+}_{0,2}(X,\beta)\right]^{\C^*, \vir} \right) .
    \end{align*}
The first equality is \eqref{e:localization}. The second line is obtained by two applications of the projection formula. 
The third line is immediate.  The forth follows from~\eqref{e: K theory equality}.

Taking the non-equivariant limit of the first and last terms in the chain of equalities and recalling Definition~\ref{Definition: Y-twisted virtual class} completes the proof.
\end{proof}


\section{Quasimap quantum Serre duality}
\label{Section: Quantum Serre duality}

In this section we prove a quantum Serre duality statement in three contexts. Specifically, we compare the twisted virtual classes
\begin{equation*}
      \left[Q^{0+}_{0,2}(X,\beta)\right]^{\vir}_{X/Z,2} \qquad \qquad \text{and} \qquad \qquad \left[Q^{0+}_{0,2}(X,\beta)\right]^{\vir}_{X/E^\vee,1}
\end{equation*}
defined in Sections~\ref{Section: Quantum Lefschetz} and~\ref{Section: The total space} to get a cycle-valued statement. From this we compute a direct relationship between the two-pointed genus-zero quasimap invariants of $Z$ and $E^\vee$. We then rephrase this as a comparison between the operators $L^{Z,0+}(z)$ and $L^{E^\vee,0+}(z)$.


\subsection{Cycle-valued statement}
\label{Section: Cycle-valued statement}

Let $\mathcal{E}\to \mathfrak{X}$ be a weakly convex vector bundle. For elements $g_1,g_2\in S$, denote by $Q_{0,g_1,g_2}^{0+}(X,\beta)$ the open and closed subset of $Q_{0,2}^{0+}(X,\beta)$ defined by the conditions $\im(ev_1)\subset X_{g_1}$ and $\im(ev_2)\subset X_{g_2}$.

\begin{lemma}
\label{Lemma: Rank computation}
On $Q_{0,g_1,g_2}^{0+}(X,\beta)$ the vector bundle $\mathbbm{R}^1\pi_*[u]^*\mathcal{E}^\vee(-x_1)$ has rank 
\begin{equation*}
    \beta(\det\mathcal{E}) - \age_{g_1}(\mathcal{E}) + \age_{g_2}(\mathcal{E}^\vee) ,
\end{equation*}
where $\det\mathcal{E}=\wedge_{i=1}^r\mathcal{E}$ is the determinant line bundle.
\end{lemma}

\begin{proof}
Fix a quasimap $(C,x_1,x_2,[u]) \in Q^{0+}_{0,g_1,g_2}(X,\beta)$ and let $r_j$ be the order of the isotropy group at $x_j$ for $j=1,2$. Recall from \S \ref{Section: Weak convexity} that there exists line bundles $\{\mathcal{L}_i\to C\}_{i=1}^r$ such that $[u]^*\mathcal{E}=\oplus_{i=1}^r \mathcal{L}_i$. Write $\mathbbm{R}^1\pi_*[u]^*\mathcal{E}^{\vee}(-x_1)$ as the direct sum $\oplus_{i=1}^r \mathbbm{R}^1\pi_*\mathcal{L}_{i}^{\vee}(-x_1)$. Using orbifold Riemann-Roch we compute:
\begin{align}
    \rank\left(\mathbbm{R}^1\pi_*\mathcal{L}_i^{\vee}(-x_1)\right) &= h^1\left( C,\mathcal{L}_i^\vee (-x_1) \right) \nonumber \\
    &= h^0\left( C,\mathcal{L}_i (-x_2) \right) \nonumber \\
    &= \deg(\mathcal{L}_i(-x_2)) + 1 - \age_{x_1}(\mathcal{L}_i) - \age_{x_2}(\mathcal{L}_i(-x_2)) \nonumber \\
    &= \deg(\mathcal{L}_i) - \frac{1}{r_2} + 1 - \age_{x_1}(\mathcal{L}_i) - \age_{x_2}(\mathcal{L}_i(-x_2))  , \label{e: rank calc eqn 1}
\end{align}
where $\age_{x_j}(-)$ denotes the age with respect to cyclic generator which acts on a local chart by multiplication by $e^{2\pi i\frac{1}{r_j}}$. The second equality is Serre duality and Lemma~\ref{Lemma: The log canonical bundle over a given source curve is trivial}.  The third equality follows from assuming $\mathcal{E}$ is weakly convex.

We claim that
\begin{equation}
\label{e: rank calc eqn 2}
    - \frac{1}{r_2} + 1 - \age_{x_2}(\mathcal{L}_i(-x_2))
    =
    \age_{x_2}(\mathcal{L}_i^\vee) .
\end{equation}
To see this, consider three cases. 

First, assume $G_{x_2}$ acts nontrivially on the fiber $\mathcal{L}_i(-x_2)|_{x_2}$. Then we observe that
\begin{align*}
    - \frac{1}{r_2} + 1 - \age_{x_2}(\mathcal{L}_i(-x_2))
    &= - \frac{1}{r_2} + \age_{x_2}(\mathcal{L}_i^\vee(x_2)) \\
    &= \age_{x_2}(\mathcal{L}_i^\vee) ,
\end{align*}
where the second equality follows because the age of $\mathcal{O}_{C}(x_2)$ at $x_2$ is $1/r_2$.

For the second case, assume $G_{x_2}$ is nontrivial but acts trivially on the fiber $\mathcal{L}_{i}(-x_2)|_{x_2}$. Then the age of $\mathcal{L}_{i}$ and $\mathcal{L}_{i}^\vee$ at $x_2$ is $1/r_2$ and $(r_2-1)/r_2$ respectively. We have the following:
\begin{align*}
    - \frac{1}{r_2} + 1 - \age_{x_2}(\mathcal{L}_{_i}(-x_2))
    &= \frac{r_2 - 1}{r_2} \\
    &= \age_{x_2}(\mathcal{L}_{i}^\vee) .
\end{align*}

Finally, consider the case $g_2$ equals the identity. Then $G_{x_2}$ is trivial, hence $r_2$ equals $1$. We obtain
\begin{align*}
    - \frac{1}{r_2} + 1 - \age_{x_2}(\mathcal{L}_{i}(-x_2)) = 0 = \age_{x_2}(\mathcal{L}_{i}^\vee) .
\end{align*}
This proves the claim. 

Note that
\begin{equation*}
    \sum_{i=1}^r \deg(\mathcal{L}_i) = \deg( [u]^*(\det \mathcal{E}) ) = \beta( \det\mathcal{E} )
\end{equation*}
and
\begin{equation*}
    \sum_{i=1}^r \age_{x_j}(\mathcal{L}_i) = \age_{x_j}([u]^*\mathcal{E}) = \age_{g_j}(\mathcal{E}) .
\end{equation*}
Summing \eqref{e: rank calc eqn 1} and \eqref{e: rank calc eqn 2} over $i$ then completes the proof.
\end{proof}

We conclude this section with the following equality of twisted virtual classes. This may be interpreted as a cycle-valued formulation of quantum Serre duality.

\begin{theorem}\emph{[Cycle-valued quantum Serre duality.]}
\label{Theorem: Cycle-valued QSD}
Let $\mathcal{E}$ be weakly convex. On the connected component $Q_{0,g_1,g_2}^{0+}(X,\beta)$, the virtual classes $\left[Q^{0+}_{0,g_1,g_2}(X,\beta)\right]^{\vir}_{X/Z,2}$ and $\left[Q^{0+}_{0,g_1,g_2}(X,\beta)\right]^{\vir}_{X/E^\vee,1}$ are equal up to sign,
\begin{equation*}
    \left[Q^{0+}_{0,g_1,g_2}(X,\beta)\right]^{\vir}_{X/Z,2} = (-1)^{\beta(\det \mathcal{E})- \age_{g_1}(\mathcal{E}) + \age_{g_2}(\mathcal{E}^\vee)} \left[Q^{0+}_{0,g_1,g_2}(X,\beta)\right]^{\vir}_{X/E^\vee,1} .
\end{equation*}
\end{theorem}

\begin{proof}
We defined the twisted virtual classes in~\ref{Definition: Z-twisted virtual class} and~\ref{Definition: Y-twisted virtual class} as
\begin{align*}
    \left[Q^{0+}_{0,g_1,g_2}(X,\beta)\right]^{\vir}_{X/Z,2} &:= e(\mathbbm{R}^0\pi_*[u]^*\mathcal{E}(-x_2)) \cap \left[Q^{0+}_{0,g_1,g_2}(X,\beta)\right]^{\vir} \\
    \left[Q^{0+}_{0,g_1,g_2}(X,\beta)\right]^{\vir}_{X/E^\vee,1} &:= e(\mathbbm{R}^1\pi_*[u]^*\mathcal{E^\vee}(-x_1)) \cap \left[Q^{0+}_{0,g_1,g_2}(X,\beta)\right]^{\vir} .
\end{align*}

Via Proposition~\ref{Proposition: The log canonical bundle over the universal curve is trivial}, 
\begin{equation}
    \label{Equation: Proof of cycle-valued QSD - log canonical bundle is trivial}
    \mathcal{O}_\mathcal{C} = \omega_\pi(x_1+x_2).
\end{equation}
Tensoring \eqref{Equation: Proof of cycle-valued QSD - log canonical bundle is trivial} by $[u]^*\mathcal{E}(-x_2)$, pushing forward via $\pi_*$, and applying Serre duality yields:
\begin{equation*}
    \mathbbm{R}^0\pi_*[u]^*\mathcal{E} (-x_2) = \left( \mathbbm{R}^1\pi_* [u]^*\mathcal{E}^{\vee} (-x_1) \right)^\vee.
\end{equation*}
This implies the Euler class identity
\begin{align*}
    e\left( \mathbbm{R}^0\pi_*[u]^*\mathcal{E} (-x_2) \right) &= e\left( \left( \mathbbm{R}^1\pi_* [u]^*\mathcal{E}^{\vee} (-x_1) \right)^\vee \right) \\
    &= (-1)^{\beta(\det \mathcal{E})- \age_{g_1}(\mathcal{E}) + \age_{g_2}(\mathcal{E}^\vee)} e\left( \mathbbm{R}^1\pi_* [u]^*\mathcal{E}^{\vee} (-x_1) \right) ,
\end{align*}
where the second equality follows from Lemma~\ref{Lemma: Rank computation}.
\end{proof}


\subsection{Quantum Serre duality for quasimap invariants}
\label{Section: Quantum Serre duality for quasimap invariants}
In this section we use the comparison of twisted virtual cycles of Theorem~\ref{Theorem: Cycle-valued QSD} to compare quasimap invariants for $Z$ and $E^\vee$ when $\mathcal{E}$ is weakly convex. We obtain a simple relationship between the generating functions $L^{Z,0+}(z)$ and $L^{E^\vee,0+}(z)$.

For the remainder of the paper we assume: 
\begin{itemize}
    \item Assumption~\ref{Assumption: The Poincare pairing on the ambient cohomology of Z is non-degenerate};
    \item the GIT stack quotient $X$ is proper;
    \item the vector bundle $\mathcal{E}$ is weakly convex.
\end{itemize}
In particular, the third case holds whenever $\mathcal{E}$ is weakly semi-positive.

\begin{definition}\parencite[Definition 6.9]{Shoemaker_2018}
Given a class $\gamma\in H^*_{\CR, \ct}(E^\vee)$, denote by $\overline{\gamma}$ a lift of $\gamma$ to the compactly supported cohomology. Define the linear map $\Delta:H_{\CR,\ct}^*(E^\vee)\to H_{\CR,\amb}^*(Z)$ by
\begin{equation*}
    \Delta(\gamma) := j^* \circ p_*^{\cs}(\overline{\gamma}) ,
\end{equation*}
where $p_*^{\cs}: H^*_{\CR,\cs}(E^\vee)\to H^*_{\CR}(X)$ is the pushforward of compactly supported cohomology.
\end{definition}

\begin{lemma} \emph{ \parencite[Lemma 6.10]{Shoemaker_2018} }
Assuming~\ref{Assumption: The Poincare pairing on the ambient cohomology of Z is non-degenerate}, the transformation 
\begin{equation*}
    \Delta:H_{\CR,\ct}^*(E^\vee)\to H_{\CR,\amb}^*(Z)
\end{equation*}
is well defined.
\end{lemma}

\begin{proof}
We work over a given twisted sector $X_g$ for $g\in S$. Let $\gamma$ be an element of $H_{\ct}^*(E^\vee_g)$. A lift $\overline{\gamma} \subset H^*(E^\vee_g)$ is defined up to an element of the kernel of $\phi_g:H^*_{\cs}(E^\vee_g)\to H^*(E^\vee_g)$. To show that $\Delta$ is well defined, we must show $p^{\cs}_{g*}(\ker(\phi_g)) \subseteq \ker(j^*_g)$.

Let $\kappa$ be an element of $\ker(\phi_g)$. The pushforward $i_{g*}^{\cs}:H^*(X_g)\to H_{\cs}^*(E^\vee_g)$ is an isomorphism. By the compactness of $X$, there exists $\alpha\in H^*(X_g)$ such that $i_{g*}^{\cs}(\alpha)=\kappa$. Note that $p_{g*}^{\cs}(\kappa)=p_{g*}^{\cs} \circ i_{g*}^{\cs} (\alpha)=\alpha$. We see that
\begin{equation*}
    0 = \phi_g(\kappa) = \phi_g\circ i^{\cs}_{g*}(\alpha) = i_{g*}(\alpha) = p^*(e(E^\vee_g) \cup \alpha) ,
\end{equation*}
where the last equality is \eqref{Equation: Rewriting an element of narrow cohomology}. Assumption~\ref{Assumption: The Poincare pairing on the ambient cohomology of Z is non-degenerate} implies $j^*_g(\alpha)=0$ if and only if $j_{g*} \circ j^*_g(\alpha)=0$. The projection formula implies $j_{g*} \circ j^*_g(\alpha)=e(E_g)\cup \alpha$. Up to a sign this is equal to
\begin{equation*}
    e(E^\vee_g) \cup \alpha = i^*_g( p^*_g(e(E^\vee_g) \cup \alpha) ) = 0 .
\end{equation*}
Therefore $p_{g*}^{\cs}(\kappa)=\alpha$ lies in $\ker(j^*_g)$.
\end{proof}

Recall Lemma~\ref{Lemma: Narrow cohomology isomorphism}, which states that any element $\gamma$ in $H^*_{\ct}(E^\vee_g)$ may be written as $p^*_g( e(E^\vee_g) \cup \alpha )$ for some $\alpha$ in $H^*(X_g)$.

\begin{lemma}\emph{ \parencite[Lemma 6.11]{Shoemaker_2018} }
\label{Lemma: Delta_+ is an isomorphism}
Given $\gamma\in H^*_{\ct}(E^\vee_g)$, choose $\alpha\in H^*(X_g)$ such that $\gamma = p_g^*( e(E^\vee_g) \cup \alpha )$, then
\begin{equation*}
    \Delta(\gamma) = j^*_g(\alpha) .
\end{equation*}
Furthermore, the transformation $\Delta:H^*_{\CR,\ct}(E^\vee) \to H^*_{\CR,\amb}(Z)$ is an isomorphism.
\end{lemma}

\begin{proof}
If $g$ fixes only the origin of $W\times \C^r$, then $E^\vee_g$ equals $X_g$ and the result is immediate. When that is not the case, the proof is similar to that in Lemma~6.11 of \cite{Shoemaker_2018}.

By Equation~\eqref{Equation: Rewriting an element of narrow cohomology}, we have
\begin{equation*}
    \gamma = p^*_g\left( e(E_g^\vee) \cup \alpha \right) = i_{g*} ( \alpha ) . 
\end{equation*}
Noting that $i_{g*}$ factors as $\phi \circ i_{g*}^{\cs}$ gives the following:
\begin{align*}
    \Delta( \gamma ) &= \Delta \circ \phi_g \circ i_{g*}^{\cs} (\alpha) \\
    &= j^*_g \circ p_{g*}^{\cs} \circ i_{g*}^{\cs} (\alpha) \\
    &= j^*_g (\alpha).
\end{align*}

To prove the second claim, we observe
\begin{align*}
    H^*_{\ct}(E^\vee_g) &= p^*_g\left( \im( e(E^\vee_g) \cup - ) \right) \\
    &= \im( e(E^\vee_g) \cup - ) \\
    &= \im( e(E_g) \cup - ) \\
    &\cong j_{g*}\left( \im( j^*_g ) \right) \\
    &\cong \im( j^*_g ) \\
    &= H^*_{\amb}(Z_g) .
\end{align*}
The first equality follows from Lemma~\ref{Lemma: Narrow cohomology isomorphism}, the isomorphism in line four is from the projection formula, and the fifth isomorphism is from Assumption~\ref{Assumption: The Poincare pairing on the ambient cohomology of Z is non-degenerate}. 
\end{proof}

It will be useful to consider a modification of $\Delta$.

\begin{definition}
Define the linear transformation $\tilde\Delta$ by
\begin{equation*}
    \tilde\Delta|_{H^*_{\ct}(E^\vee_g)} := e^{\pi i \age_g(\mathcal{E})} \Delta |_{H^*_{\ct}(E^\vee_g)} . 
\end{equation*}
\end{definition}

\begin{lemma}
\label{Lemma: Pairing comparison}
The transformation $\tilde \Delta$ identifies the Chen-Ruan Poincar\'e pairings up to a sign:
\begin{equation*}
    \langle \tilde\Delta(\gamma_1), \tilde\Delta(\gamma_2) \rangle^Z = (-1)^{\rank(E)} \langle \gamma_1, \gamma_2 \rangle^{E^\vee, \ct} ,
\end{equation*}
for any $\gamma_1$ and $\gamma_2$ in $H^*_{\CR,\ct}(E^\vee)$.
\end{lemma}

\begin{proof}
Assume $\gamma_1$ is supported on $E^\vee_{g_1}$ and $\gamma_2$ is supported on $E^\vee_{g_2}$. The pairings $\langle \tilde\Delta(\gamma_1), \tilde\Delta(\gamma_2) \rangle^Z$ and $\langle \gamma_1, \gamma_2 \rangle^{E^\vee, \ct}$ equal zero unless $\iota(E^\vee_{g_1})=E^\vee_{g_2}$. 

It therefore suffices to consider the case $\gamma_1\in H^*_{\ct}(E^\vee_{g})$ and $\gamma_2\in H^*_{\ct}(E^\vee_{g^{-1}})$. Choose $\alpha_1, \alpha_2 \in H^*_{\CR}(X)$ such that $\gamma_1=i_{g*}(\alpha_1)$ and $\gamma_2=i_{g^{-1}*}(\alpha_2)$. Then, applying the projection formula and Lemma~\ref{Lemma: Delta_+ is an isomorphism}, we obtain
\begin{align*}
    \langle \tilde\Delta(\gamma_1), \tilde\Delta(\gamma_2) \rangle^Z &= (-1)^{\age_{g}(\mathcal{E}) + \age_{g^{-1}}(\mathcal{E})} \int_{IZ} j^*_g(\alpha_1) \cup \iota^*( j^*_{g^{-1}}(\alpha_2) ) \\
    &= (-1)^{\rank(E)-\rank(E_{g})} \int_{IX} \alpha_1 \cup \iota^* ( \alpha_2 \cup e(E_{g}) ) \\
    &= (-1)^{\rank(E)} \int_{IX} \alpha_1 \cup \iota^* ( i^*_{g^{-1}} i_{g^{-1}*}(\alpha_2) ) \\
    &= (-1)^{\rank(E)} \int_{IE^\vee} i_{g*}^{\cs}(\alpha_1) \cup \iota^* ( i_{g^{-1}*}(\alpha_2) ) \\
    &= (-1)^{\rank(E)} \langle \gamma_1, \gamma_2 \rangle^{E^\vee, \ct} ,
\end{align*}
where the second equality uses the fact \parencite[Lemma~4.6]{Adem_Leida_Ruan_Orbifolds_and_stringy_topology_2007} that for all $g$ in $S$
\begin{equation}
\label{Equation: Sum of ages}
    \age_{g}(\mathcal{E}) + \age_{g}(\mathcal{E}^\vee) = \rank(E) - \rank(E_{g}) .
\end{equation}
\end{proof}

\begin{theorem}
\label{Theorem: Two-pointed invariant relationship}
Given elements $\gamma_1,\gamma_2\in H^*_{\CR,\ct}(E^\vee)$, we have the equality 
\begin{equation*}
    \left\langle \tilde\Delta\left(\gamma_1\right) \psi_1^{a_1}, \tilde\Delta\left(\gamma_2\right) \psi_2^{a_2} \right\rangle^{Z,0+}_{0,\beta}
    =
    e^{\pi i (\beta(\det \mathcal{E})+\rank(E))} \left\langle \gamma_1 \psi_1^{a_1}, \gamma_2 \psi_2^{a_2} \right\rangle^{E^\vee,0+}_{0,\beta} .
\end{equation*}
\end{theorem}

\begin{proof}
Let $g_1$ and $g_2$ be elements in $S$. Without loss of generality, choose $\gamma_1\in H^*_{\ct}(E^\vee_{g_1})$, $\gamma_2\in H^*_{\ct}(E^\vee_{g_2})$.  Choose $\alpha_1\in H^*(X_{g_1})$, and $\alpha_2\in H^*(X_{g_2})$ such that
\begin{align*}
    \gamma_1 &= p^*_{g_1}(\alpha_1 \cup e(E^\vee_{g_1})) = i_{g_1 *}(\alpha_1) \\
    \gamma_2 &= p^*_{g_2}(\alpha_2 \cup e(E^\vee_{g_2})) = i_{g_2 *}(\alpha_2) .
\end{align*}

By Proposition~\ref{Proposition: Convexity for orbifold hypersurfaces}, our assumption that $\mathcal{E}$ is weakly convex, and \eqref{e: restricted to x1} there is a short exact sequence
\begin{equation}
\label{Equation: SES for invariant relation}
\begin{tikzcd}
    0 \arrow[r]
    & \mathbbm{R}^0\pi_*[u]^*\mathcal{E}(-x_2) \arrow[r]
    & \mathbbm{R}^0\pi_*[u]^*\mathcal{E} \arrow[r]
    & ev_2^{X*}(E_{g_2}) \arrow[r]
    & 0 .
\end{tikzcd}
\end{equation}
Also note that the following diagram commutes:
\begin{equation}
\label{comD - eval maps for Q(Z) and Q(X)}
\begin{tikzcd}
    Q_{0,2}^{0+}(Z,\beta) \arrow[r,"\tilde{j}"] \arrow[d, "ev_i^{Z}"]
    & Q_{0,2}^{0+}(X,\beta) \arrow[d, "ev_i^{X}"] \\
    \bar IZ \arrow[r, "j"]
    & \bar IX .
\end{tikzcd}
\end{equation}

Expanding the integral in the left-hand-side of the statement gives us
\begin{align}
    &\left\langle \tilde\Delta\left(\gamma_1\right) \psi_1^{a_1}, \tilde\Delta\left(\gamma_2\right) \psi_2^{a_2} \right\rangle^{Z,0+}_{0,\beta} 
    \label{Equation: Two-pointed invariant of Z} \\
    =& e^{\pi i (\age_{g_1}(\mathcal{E})+\age_{g_2}(\mathcal{E}))} \int_{\left[Q_{0,2}^{0+}(Z,\beta)\right]^{\vir}} ev_1^{Z *} j^*_{g_1}(\alpha_1)  \psi_1^{a_1} \cup ev_2^{Z *} j^*_{g_2}(\alpha_2)  \psi_2^{a_2} \nonumber \\
    =& e^{\pi i (\age_{g_1}(\mathcal{E})+\age_{g_2}(\mathcal{E}))} \int_{\left[Q_{0,2}^{0+}(X,\beta)\right]^{\vir}} ev_1^{X *} (\alpha_1) \psi_1^{a_1} \cup ev_2^{X *} (\alpha_2) \psi_2^{a_2} \cup e(\mathbbm{R}^0\pi_*[u]^*\mathcal{E}) \nonumber \\
    =& e^{\pi i (\age_{g_1}(\mathcal{E})+\age_{g_2}(\mathcal{E}))} \int_{\left[Q_{0,2}^{0+}(X,\beta)\right]^{\vir}} ev_1^{X *} (\alpha_1) \psi_1^{a_1} \cup ev_2^{X *} (\alpha_2 \cup e(E_{g_2})) \psi_2^{a_2} \cup e(\mathbbm{R}^0\pi_*[u]^*\mathcal{E}(-x_2)) \nonumber \\
    =& e^{\pi i (\age_{g_1}(\mathcal{E})+\age_{g_2}(\mathcal{E}) + \rank(E_{g_2}))} \int_{\left[Q_{0,2}^{0+}(X,\beta)\right]^{\vir}_{X/Z,2}} ev_1^{X *} (\alpha_1) \psi_1^{a_1} \cup ev_2^{X *} (\alpha_2 \cup e(E_{g_2}^\vee)) \psi_2^{a_2} \nonumber \\
    =& e^{\pi i (\beta(\det \mathcal{E})+\rank(E))} \int_{\left[Q_{0,2}^{0+}(X,\beta)\right]^{\vir}_{X/E^\vee,1}} ev_1^{X *} (\alpha_1) \psi_1^{a_1} \cup ev_2^{X *} (\alpha_2 \cup e(E_{g_2}^\vee)) \psi_2^{a_2}  \nonumber \\
    =& e^{\pi i (\beta(\det \mathcal{E})+\rank(E))}  \left\langle \gamma_1 \psi_1^{a_1}, \gamma_2 \psi_2^{a_2} \right\rangle^{E^\vee,0+}_{0,\beta}. \nonumber
\end{align}
The second equality follows from \eqref{comD - eval maps for Q(Z) and Q(X)}, the projection formula, and Proposition~\ref{Proposition: Quantum Lefschetz}. The third equality is by \eqref{Equation: SES for invariant relation}. The fifth equality is by \eqref{Equation: Sum of ages} and Theorem~\ref{Theorem: Cycle-valued QSD}.  The final equality is by Proposition~\ref{prop: this fits here}.
\end{proof}

\begin{theorem}
\label{Theorem: QSD for quasimaps}
The transformation $\tilde\Delta$ identifies the operators $L^{Z,0+}(z)$ and $L^{E^\vee,0+}(z)$ up to a change of variables in the Novikov parameter:
\begin{equation*} 
    L^{Z,0+}(z) \circ \tilde\Delta
    =
    \tilde\Delta \circ L^{E^\vee,0+}(z) |_{q^\beta\mapsto e^{\pi i \beta(\det \mathcal{E})} q^\beta} .
\end{equation*}
\end{theorem}

\begin{proof}
By Assumption~\ref{Assumption: The Poincare pairing on the ambient cohomology of Z is non-degenerate}, the Chen-Ruan Poincar\'e pairing on the ambient cohomology of $Z$ is non-degenerate and by Lemma~\ref{Lemma: Delta_+ is an isomorphism}, $\Delta$ is an isomorphism. It therefore suffices to show that
\begin{equation}
    \label{e: pairing L-operators on Z}
    \left\langle \tilde\Delta(\gamma_2), L^{Z,0+}(z)\circ \tilde\Delta(\gamma_1) \right\rangle^Z 
    =
    \left\langle \tilde\Delta(\gamma_2), \tilde\Delta \circ L^{E^\vee,0+}(z)(\gamma_1) |_{q^\beta \mapsto e^{\pi i \beta(\det \mathcal{E})} q^\beta} \right\rangle^Z ,
\end{equation}
for any $\gamma_1$ and $\gamma_2$ in $H^*_{\CR,\ct}(E^\vee)$.

Assume $\gamma_1$ is in $H^*_{\ct}(E^\vee_{g_1})$ and $\gamma_2$ is in $H^*_{\ct}(E^\vee_{g_2})$. Expand the left-hand-side of \eqref{e: pairing L-operators on Z} to obtain
\begin{align}
    &\langle \tilde\Delta(\gamma_2), \tilde\Delta(\gamma_1) \rangle^Z + \nonumber \\
    &\sum_{\beta \in \Eff} q^\beta \left\langle \tilde\Delta(\gamma_2) , \tilde{ev}_{2*}\left( \frac{ev_1^*\left(\tilde\Delta(\gamma_1)\right)}{-z-\psi_1} \cap \left[ Q^{0+}_{0,2}(Z,\beta) \right]^{\vir} \right) \right\rangle^Z .
    \label{e: pairing with the Z-operator}
\end{align}
The right-hand-side of \eqref{e: pairing L-operators on Z} expands as
\begin{align}
     &\langle \tilde\Delta(\gamma_2), \tilde\Delta(\gamma_1) \rangle^Z + \nonumber \\
     &\sum_{\beta \in \Eff} e^{\pi i \beta(\det \mathcal{E})} q^\beta \left\langle \tilde\Delta(\gamma_2) , \tilde\Delta \circ \tilde{ev}_{2*}\left( \frac{ev_1^*(\gamma_1)}{-z-\psi_1} \cap \left[ Q^{0+}_{0,2}(E^\vee,\beta) \right]^{\vir} \right) \right\rangle^Z .
     \label{e: pairing with E^vee-operator}
\end{align}

For $\beta\in \Eff(W,G,\theta)$, we have
\begin{align*}
    &\left\langle \tilde\Delta(\gamma_2) , \tilde{ev}_{2*}\left( \frac{ev_1^*\left(\tilde\Delta(\gamma_1)\right)}{-z-\psi_1} \cap \left[ Q^{0+}_{0,2}(Z,\beta) \right]^{\vir} \right) \right\rangle^Z \\
    =&\left\langle \frac{\tilde\Delta(\gamma_1)}{-z-\psi_1} , \tilde\Delta(\gamma_2) \right\rangle_{0,\beta}^{Z,0+} \\
    =& e^{\pi i (\beta(\det \mathcal{E})+\rank(E))} \left\langle \frac{\gamma_1}{-z-\psi_1} , \gamma_2 \right\rangle_{0,\beta}^{E^\vee,0+} \\
    =& e^{\pi i (\beta(\det \mathcal{E})+\rank(E))} \left\langle \gamma_2 , \tilde{ev}_{2*}\left( \frac{ev_1^*(\gamma_1)}{-z-\psi_1} \cap \left[ Q^{0+}_{0,2}(E^\vee,\beta) \right]^{\vir} \right) \right\rangle^{E^\vee,\ct} \\
    =& e^{\pi i \beta(\det \mathcal{E})} \left\langle \tilde\Delta(\gamma_2) , \tilde\Delta \circ \tilde{ev}_{2*}\left( \frac{ev_1^*(\gamma_1)}{-z-\psi_1} \cap \left[ Q^{0+}_{0,2}(E^\vee,\beta) \right]^{\vir} \right) \right\rangle^{Z} .
\end{align*}
The first equality is the projection formula. The second is Theorem~\ref{Theorem: Two-pointed invariant relationship}. The third is another application of the projection formula. The fourth is Lemma~\ref{Lemma: Pairing comparison}.

Each coefficient of $q^\beta$ in \eqref{e: pairing with the Z-operator} and \eqref{e: pairing with E^vee-operator} is equal, completing the proof.
\end{proof}

\subsection{Quantum Serre duality for Gromov-Witten invariants}

In this section we combine the above results with the wall-crossing formulas proven by Zhou in \cite{zhou2020quasimap} to prove a quantum Serre duality statement for Gromov-Witten invariants without assuming convexity.

Theorem~1.12.2 of \cite{zhou2020quasimap} states
\begin{equation}
\label{Equation: Zhou wall-crossing}
    J^{X,\infty}(\boldsymbol{t}+\mu^{X,\geq 0+}(q,-z),q,z) = J^{X,0+}(\boldsymbol{t},q,z),
\end{equation}
where $J^{X,\varepsilon}(\boldsymbol{t},q,z)$ is defined in \parencite[\S 1.12]{zhou2020quasimap} and $\mu^{X,\geq 0+}(q,-z)$ is defined in \parencite[\S 1.11]{zhou2020quasimap}. Similar to the $J^{X,\varepsilon}$-function defined in \cite{zhou2020quasimap}, we may generalize the operator $L^{X,\varepsilon}(\boldsymbol{t},z)$ of \eqref{Equation: General Fundamental Solution} by including $\psi$-classes in the $\boldsymbol{t}$ insertions of \eqref{equation: double brackets} and \eqref{Equation: General Fundamental Solution}. Let $f$ be a cohomology-valued polynomial. We redefine the double-bracket as
\begin{align*}
    &\langle \langle \alpha_1 \psi_1^{a_1}, \ldots , \alpha_k \psi_k^{a_k} \rangle \rangle_{0}^{X,\varepsilon} (f(\psi)) \\
    :=& \sum_{\beta\in\Eff} \sum_{m\geq 0} \frac{q^\beta}{m!} \langle \alpha_1 \psi_1^{a_1} , \ldots , \alpha_k \psi_k^{a_k} , f(\psi_{k+1}) , \ldots , f(\psi_{k+m}) \rangle_{0,\beta}^{X,\varepsilon}
\end{align*}
for classes $\alpha_1,\ldots,\alpha_k \in H^*_{\CR}(X)$ and non-negative integers $a_1,\ldots,a_k$.  Then define $L^{X,\varepsilon}(f(\psi),z)$ as in \eqref{Equation: General Fundamental Solution} by replacing $\boldsymbol{t}$ with $f(\psi)$.  We now use \eqref{Equation: Zhou wall-crossing} to relate $L^{X,0+}(z) $ and $L^{X,\infty}(\mu^{X,\geq 0+}(q,-\psi),z)$.

\begin{lemma}
\label{Lemma: Wall-crossing}
The operators $L^{X,0+}(z) $ and $L^{X,\infty}(\mu^{X,\geq 0+}(q,-\psi),z)$ are equal.
\end{lemma}

\begin{proof}
Let $\{T_i\}_{i\in I}$ be a basis of $H^*_{\CR}(X)$ with dual basis $\{T^i\}_{i\in I}$. For $p\in I$, observe that
\begin{equation}\label{Equation: Derivative of quasimap J-function}
    z\frac{\partial}{\partial t_p} J^{X,0+}(\boldsymbol{t},q,z)|_{t=0} 
    = T_p + \sum_{\beta\in \Eff} \sum_{i\in I} q^\beta \left\langle T_p , \frac{T^i}{z-\psi} \right\rangle_{0,\beta}^{X,0+} T_i .
\end{equation}
Pairing \eqref{Equation: Derivative of quasimap J-function} with $\alpha \in H^*_{\CR}(X)$, we obtain
\begin{align} \label{e:afterpairing1}
    \left\langle z\frac{\partial}{\partial t_p} J^{X,0+}(\boldsymbol{t},q,z)|_{t=0} , \alpha \right\rangle^X &= \langle T_p, \alpha \rangle + \sum_{\beta\in\Eff} q^{\beta} \left\langle T_p, \frac{\alpha}{z-\psi} \right\rangle_{0,\beta}^{X,0+} \\ \nonumber
    &= \left\langle T_p , L^{X,0+}(-z)(\alpha) \right\rangle^X.
\end{align}

Similarly, note that
\begin{align}
    &z\frac{\partial}{\partial t_p} J^{X,\infty}(\boldsymbol{t}+\mu^{X,\geq 0+}(q,-z),q,z)|_{t=0} \label{Equation: Derivative of GWT J-function} \\
    =& T_p + \sum_{i\in I} \sum_{\beta\in \Eff} \sum_{m\geq 0} \frac{q^\beta}{m!} \left\langle T_p , \frac{T^i}{z-\psi} , \mu^{X,\geq 0+}(q,-\psi) , \ldots , \mu^{X,\geq 0+}(q,-\psi) \right\rangle_{0,m+2,\beta}^{X,\infty} T_i \nonumber
\end{align}
and
\begin{align} \label{e:afterpairing2}
    &\left\langle z\frac{\partial}{\partial t_p} J^{X,\infty}(\boldsymbol{t}+\mu^{X,\geq 0+}(q,-z),q,z)|_{t=0} , \alpha \right\rangle^X \\ \nonumber
    =& \langle T_p, \alpha \rangle + \sum_{\beta\in\Eff} \sum_{m\geq 0} \frac{q^{\beta}}{m!} \left\langle T_p , \frac{\alpha}{z-\psi} , \mu^{X,\geq 0+}(q,-\psi) , \ldots , \mu^{X,\geq 0+}(q,-\psi) \right\rangle_{0,m+2,\beta}^{X,\infty} \\ \nonumber
    =& \left\langle T_p , L^{X,\infty}(\mu^{X,\geq 0+}(q,-\psi),-z)(\alpha) \right\rangle^X . 
\end{align}
By \eqref{Equation: Zhou wall-crossing}, equations \eqref{e:afterpairing1} and \eqref{e:afterpairing2} are equal for all $p \in I$.
By the nondegeneracy of the Poincar\'e pairing on $X$, it follows that
\begin{equation*}
    L^{X,0+}(z) (\alpha) = L^{X,\infty}(\mu^{X,\geq 0+}(q,-\psi),z) (\alpha).
\end{equation*}

\end{proof}

Proposition~\ref{Proposition: Pullback of L^X/Z is equivalent to L^Z of the pullback} and Lemma~\ref{Lemma: Wall-crossing} imply the following relationship between two-pointed quasimap invariants of $X$ and Gromov-Witten invariants of $Z$.

\begin{corollary}
\label{Corollary: GW quantum Lefschetz}
The operators $L^{X/Z,0+}(z)$ and $L^{Z,\infty}(\boldsymbol{t},z)$ are related by 
\begin{equation*}
    j^* \circ L^{X/Z,0+}(z)  = L^{Z,\infty}(\mu^{Z,\geq 0+}(q,-\psi),z) \circ j^*  .
\end{equation*}
\end{corollary}

Assume that the total space of the vector bundle $E^\vee \to X$ is realized as the GIT stack quotient
$$ [W \times \C^r \sslash_\theta G].$$ In other words, we require that 
$$(W \times \C^r)^{ss} = W^{ss}(\theta) \times \C^r.$$
This guarantees that the moduli space $Q_{g,k}^{0+}(E^\vee,\beta)$ of Definition~\ref{Definition: qmaps to E dual} coincides with the space 
$Q_{g,k}^{0+}([W \times \C^r \sslash_\theta G],\beta)$ of Definition~\ref{Definition: quasimap moduli space}
(as do their respective virtual classes).  In particular, this allows us to combine the wall-crossing results of \cite{zhou2020quasimap} with Theorem~\ref{Theorem: QSD for quasimaps} to deduce the following corollary.

\begin{corollary}
\label{Corollary: GW QSD}
The operator $\tilde \Delta$ identifies the fundamental solutions $L^{Z, \infty}$ and $L^{E^\vee, \infty}$ up to a change of variables:
    \begin{align*}
        &L^{Z,\infty}(\mu^{Z,\geq 0+}(q,-\psi),z) \circ \tilde \Delta \\ 
        =& \tilde \Delta \circ L^{E^\vee,\infty}(\mu^{E^\vee,\geq 0+}(q,-\psi),z) |_{q^\beta\mapsto e^{\pi i \beta(\det \mathcal{E})} q^\beta} ,
    \end{align*}
where $\mu^{Z,\geq 0+}(q,-\psi)$ and $\mu^{E^\vee,\geq 0+}(q,-\psi)$ are the  changes of variables from \parencite[\S 1.11]{zhou2020quasimap} for $Z$ and $E^\vee$ respectively.
\end{corollary}


\printbibliography


\end{document}